\documentclass{article}
\usepackage{amsmath,amssymb,amsthm}
\usepackage{amscd}    % For the CD environment
\usepackage{graphicx}
\usepackage{enumitem}
\usepackage{float}
\usepackage{tikz}
\usetikzlibrary{arrows.meta,chains, positioning}
\usetikzlibrary{shapes.geometric, arrows}
\usepackage{hyperref}
\usepackage{graphicx}
\newtheorem{definition}{Definition}
\newtheorem{theorem}{Theorem}
\newtheorem{proposition}{Proposition}

\newtheorem{lemma}{Lemma}
\newtheorem{remark}{Remark}
\newtheorem{example}{Example}

\title{The secant quandle: an invariant of classical braids and knots}
\author{
  Seongjeong Kim\thanks{Email: \texttt{kimseongjeong@jlu.edu.cn}}%
  \quad
  Yangzhou Liu\thanks{Email: \texttt{yangzhou.liu@phystech.edu}}%
  \quad
  Vassily O. Manturov\thanks{Email: \texttt{vomanturov@yandex.ru}}
}

\begin{document}
\maketitle
\begin{abstract}
We construct a novel invariant of braids and knots, called the \textbf{secant-quandle} (\textbf{SQ}), derived from homotopy classes of generic secants and generic horizontal trisecants. This invariant provides a natural generalization of the usual knot quandle, capturing richer topological information.
\end{abstract}



\section{Introduction}

The most popular approach to classical knot theory deals with standard planar projections with crossings and arcs connecting them. In contrast, T. Fiedler and A. Mortier considered only triple points to build their elegant theory \cite{Fiedler} based upon the singularity theory of V. Vassiliev \cite{Manturov}. This establishes that trisecants play an important role in knot theory. V. O. Manturov's group $G_n^k$ directly corresponds to this perspective. He suggests considering two sets: 
\begin{itemize}
    \item the 0-dimension set $\mathcal{T}_K$ of generic trisecants (like crossings).
    \item  1-dimension set of generic secant homotopy classes $\mathcal{S}_k$ (like arcs).
\end{itemize}

A quandle originates from coloring invariants: at each crossing, there are constraints on the coloring of the arcs $a$, $b$, and $c$ that form it, where the arcs can be either of the same color or different colors. The secant-quandle introduced in this paper can be regarded as coloring each homotopy class of horizontal secant, and the horizontal trisecant plays a role similar to that of a crossing. 

The paper is organized as follows. In Section 2 we introduce some basic notation and definitions. In Section 3 we define the secant quandle for braids and knots and prove its invariance. In Section 4 we give some properties of $SQ$. In particular, we show that $SQ(L \sqcup O)$, where $O$ is the trivial knot, contains a subquandle isomorphic to the knot quandle $Q(L)$.

%%%%%%%%%%%%%%%%%%%%%%%%%%%%%%%%%%%%%%%%%%%%%%%%%%% Section 2
\section{Secant-quandle of braids and knots}
We construct a novel invariant of braids and links, \textbf{secant-quandle} (\textbf{SQ}), with generic secants serving as generators and generic horizontal trisecants serving as relations, i.e., $SQ(M)=\Gamma\langle \mathcal{S}_M\mid\mathcal{T}_M,\mathcal{I}_M\rangle$, where $M$ is a braid or link. 

Let $\beta =\{\beta_{j}\}_{j=1}^{n}$ be a braid in $\mathbb{C}\times [0,1]$. Let $j_{t} := \mathbb{C}\times \{t\} \cap \beta_{j}$, specifically, $j_0$ and $j_1$ denote the points at the top and bottom of the braid respectively. We assume that $\{j_0\}_{j=1}^n$ and $\{j_1\}_{j=1}^n$ are arranged in a sequence on the semicircle. Assume each strand is oriented, either upwards or downwards (independently for each strand).

A braid corresponds to $\binom{n}{2}$ ruled surfaces $$R_{jk}:=\{p \mid p\in  \overline{j_tk_t},t\in [0,1]\}.$$
We call an oriented line $\overline{j_tk_t}$ at $t$ from $j_{t}$ to $k_{t}$ a {\em secant}, denoted by $(j,k)_{t}$. A braid $\beta$ is \emph{generic} if each intersection point between the interior of any $R_{jl}$ and $ k_t$ is transverse at $k_{t_0}$. Then these intersection points are finite and one-to-one corresponding to \emph{generic} horizontal trisecants $(j,k,l)_{t_0}$. 

\begin{definition}\label{def1}
A \emph{film-frame} is an open disk $D$ on a ruled surface divided by two adjacent trisecants $(j^\alpha, k^\alpha, l^\alpha)$ ($\alpha=0,1$), where $\overline{D}\cap (j^\alpha, k^\alpha, l^\alpha)\in \{(j^\alpha, k^\alpha) ,(k^\alpha, l^\alpha) ,(j^\alpha, l^\alpha) \} $, as shown in Figure \ref{fig:1}. 
\end{definition}

\begin{figure}%[H]  % [H] 强制图片位置不浮动
    \centering
    \includegraphics[width=0.6\linewidth]{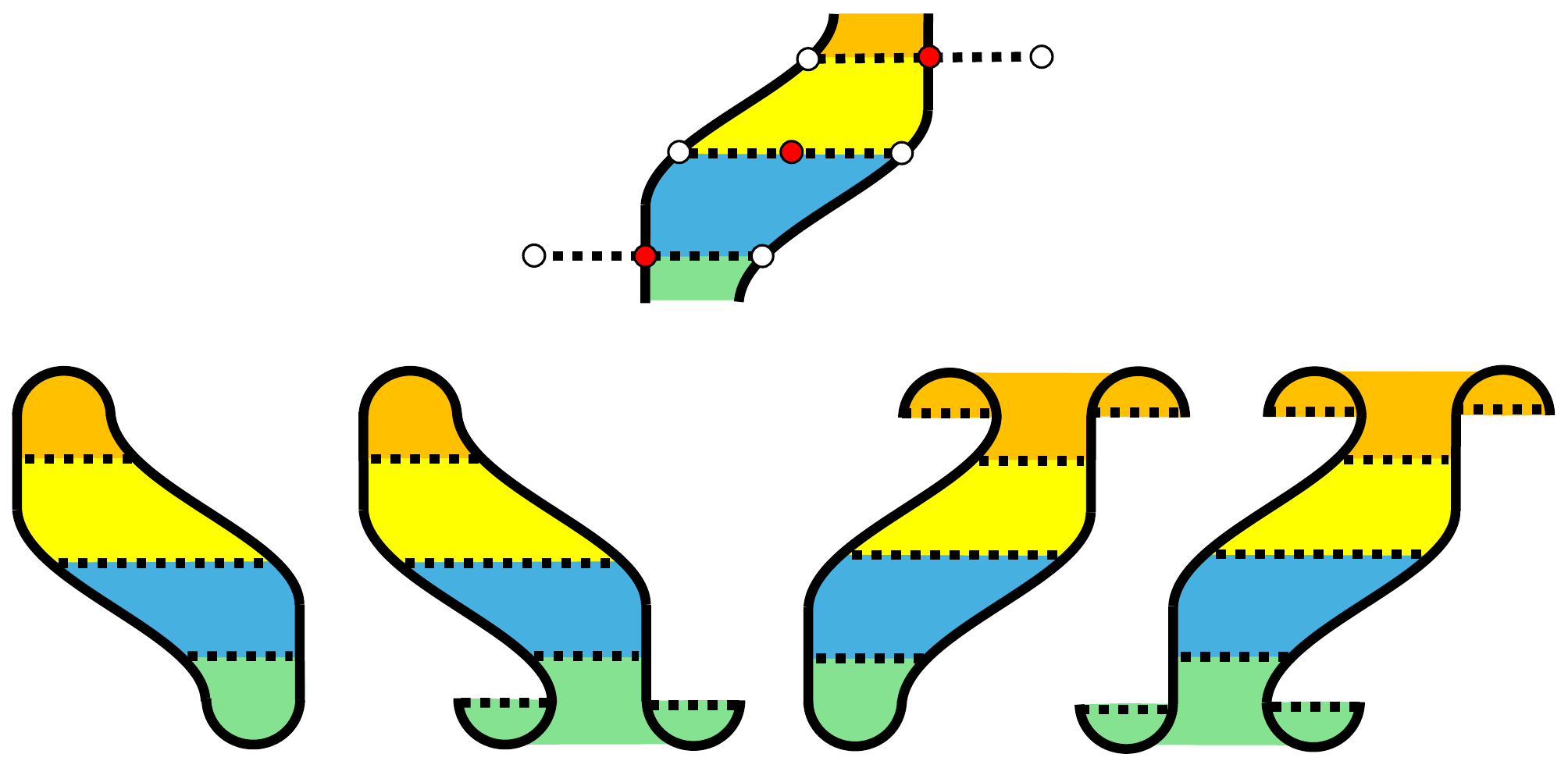}
    \caption{Trisecants divide $R_{ij}$  into several film-frames .}
    \label{fig:1}
\end{figure}

\begin{definition}\label{def2}
Two horizontal secants of a braid $\beta$ are \emph{local secant homotopic}, denoted by $(j,k) \overset{loc}{\sim} (j,k)'$, if there exist two adjacent film-frames $D_0\cup D_1\subset R_{jk}$ and a continuous family of horizontal secant $(j,k)(s)\subset  \overline{D_0\cup D_1}$, $s \in [0, 1]$ such that 
$$(j,k)(0)=(j,k) ,\quad (j,k)(1)=(j,k)'.$$  
$(j,k)\sim (j,k)'$ means that
$$(j,k)=(j^0,k^0)\overset{loc}{\sim}(j^1,k^1)\overset{loc}{\sim}\cdots \overset{loc}{\sim} (j^m,k^m)=(j,k)'.$$
\end{definition}

%\begin{remark}
%The definition can be slightly relaxed by permitting secants to degenerate to critical points at the endpoints ($t=0,1$). How does secant homotopy differ from the usual homotopy of secants? Consider a 1-punctured disc with two secants $e_\pm$, each homotopic to a top or bottom critical point of this disc, respectively. Under usual homotopy $e_+$ is equivalent to $e_-$,  whereas  secant homotopy $e_+\nsim e_-$. 
%\end{remark}

\begin{definition}\label{def3}
Let $\beta:=\{\beta_{j}\}_{j=1}^n$ be a braid. Assume that each strand is oriented.
The {\em secant quandle $SQ(\beta)$} is defined by
$$SQ(\beta):=\Gamma \langle \mathcal{S} _\beta\mid \mathcal{T}_\beta\rangle,$$
where $\mathcal{S}_\beta$ is the set of horizontal secant homotopy classes, $\mathcal{T}_\beta$ is the set of generic horizontal trisecants, and each trisecant $(j,k,l):=(j,k,l)_{t_0}$ gives a relation among nearby secants: If $\mathrm{sign} (j,k,l)_{t_0}:=\det ((j,l),Oz,T_{k}\beta)_{t_0}>0$, then we have two relations
\begin{eqnarray}
(j,k,l)&:& (j,l){_+} ={(j,l){_-}} \circ {(j,k)} \label{eq:1}\\
&& (l,j){_+}={(l,j){_-}}\mathbin{/}  {(l,k)},\label{eq:2}
\end{eqnarray}
and $(j,k){_+}\sim (j,k){_-}$, $(k,j){_+}\sim (k,j){_-}$, $(k,l){_+}\sim (k,l){_-}$, $(l,k){_+}\sim (l,k){_-}$, where $(j,k)_\pm:=(j,k)_{t_0\pm \epsilon}$.
If $\mathrm{sign} (j,k,l)<0$, then we interchange $\circ$ and $/$ in the above two equations. See Figure~\ref{figure2}. Note that if a trisecant $(j,k,l)$ has $\mathrm{sign}(j,k,l)<0$, then $\mathrm{sign}(l,k,j)>0$.
\end{definition}
\begin{figure}%[H]  % [H] 强制图片位置不浮动
    \centering
    \includegraphics[width=0.75\linewidth]{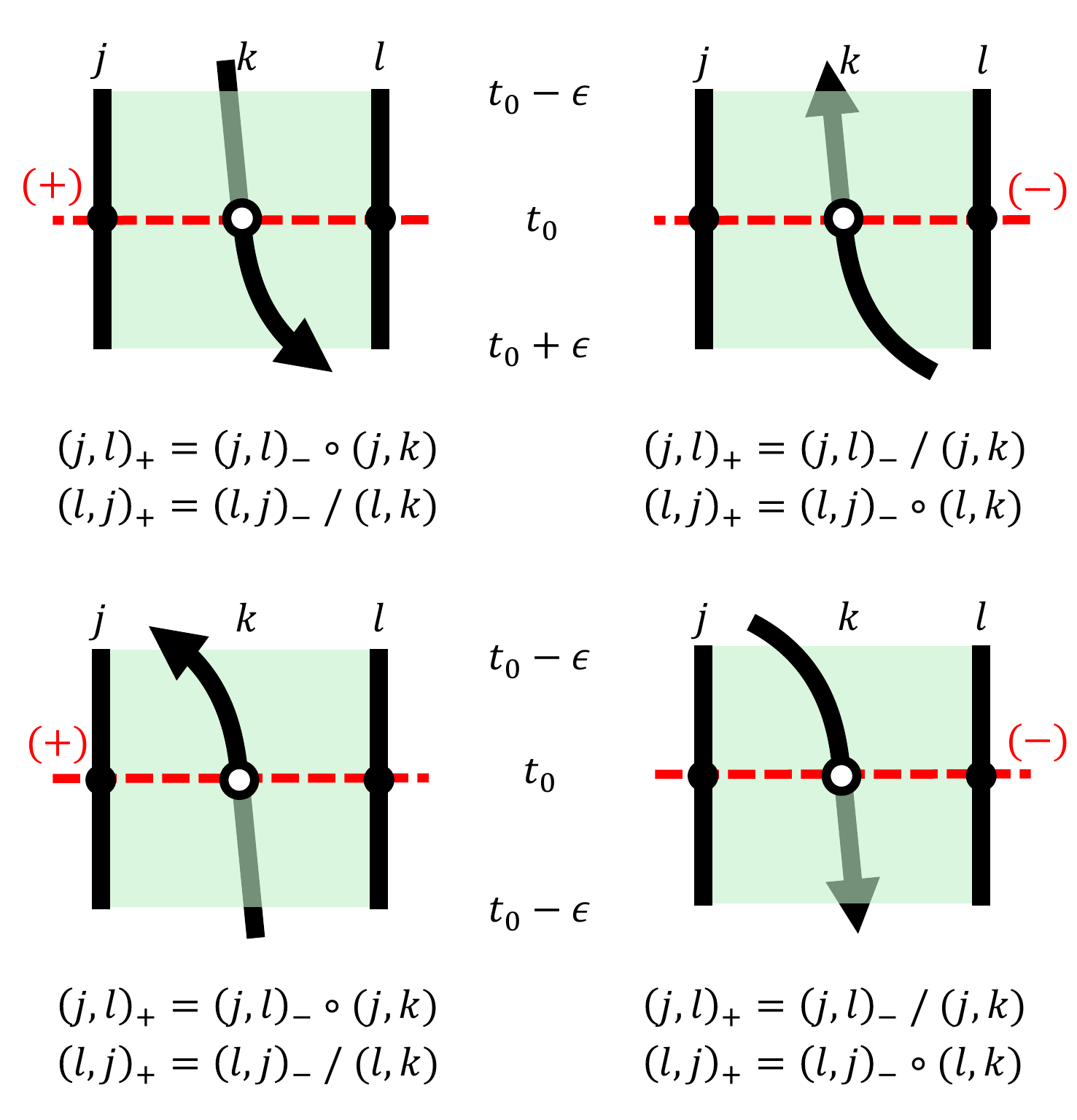}

    \caption{Relations coming from horizontal trisecants}
    \label{figure2}
\end{figure}
We will show that $SQ$ is a braid invariant in Theorem \ref{thm1}. %%%%%%%%%%%%% how about combine two sentences as one paragraph.
The definition of $SQ$ for knots is  more involved: We focus on plat closure of braids. 
\begin{definition}\label{def4}
  A knot $K:\mathbb{S}^1\hookrightarrow \mathbb{R}^3$ is a \emph{Birman} knot if $K$ is a \emph{plat closure} of a braid $\beta\in B_{2n}$ (Fig \ref{fig:2}), which has the following properties: 
  \begin{itemize}
      \item It is contained in the subset of $\mathbb{R}^3$ which is represented by the inequalities $- 0. 2 5 \leq  z \leq  1.25$.
      \item It meets the planes $z =-0.25$ and $ z=1.25$ in precisely $n$ points, and every other plane $z=z_0$, $-0.25 < z_0 < 1.25$ in precisely $2n$ points. In particular, $K \cap \mathbb{R}^{2}\times [0,1] = \beta$. We assume that strands in $K \cap \mathbb{R}^{2}\times [0,1]$ are ordered as braids and $n$ points in $K\cap \mathbb{R}^{2}\times \{-0.25\}$ ($K \cap \mathbb{R}^{2}\times \{1.25\}$ respectively) are numbered by $1,\dots,n$ from the left to the right. 
  \end{itemize}
\end{definition}

%All the knots we mentioned are based on the above braids $\beta\in B_{2n}$, and we consider their plat closures $K:=\hat{\beta}$, called \emph{Birman knot}, for which 

We have the following two propositions \cite{Birman}:

\begin{proposition}
Every link $L$ can be presented as a plat closure of a braid $\beta\in B_{2n}$. 
\end{proposition}
\begin{proposition}\label{prop2}
Let $L_1$ and $L_2$ be \emph{(stable) equivalent} links. Let $\beta_j \in B_{2n_j}$, $j=1,2$, such that the plat closure of $\beta_j$ is equivalent to $ L_j$. Then there exists $t \geq  \max(n_1, n_2)$ such that for each $n \geq t$, 
$$\beta_j'=\beta_j\sigma_{2n_j}\sigma_{2n_j+2}\cdots \sigma_{2n},\quad j=1,2$$
are in the same double coset of $B_{2n}$ modulo the subgroup $K_{2n}\subset B_{2n}$, where
$$K_{2n}:=\langle \sigma_1,\sigma_2\sigma_1^2\sigma_2,\sigma_{2k}\sigma_{2k-1}\sigma_{2k+1}\sigma_{2k}, 1\leq k\leq n-1\rangle.$$
The generators of $K_{2n}$ depicted in Figure \ref{fig:2}. 

\begin{figure}%[H]  % [H] 强制图片位置不浮动
    \centering
    \includegraphics[width=0.75\linewidth]{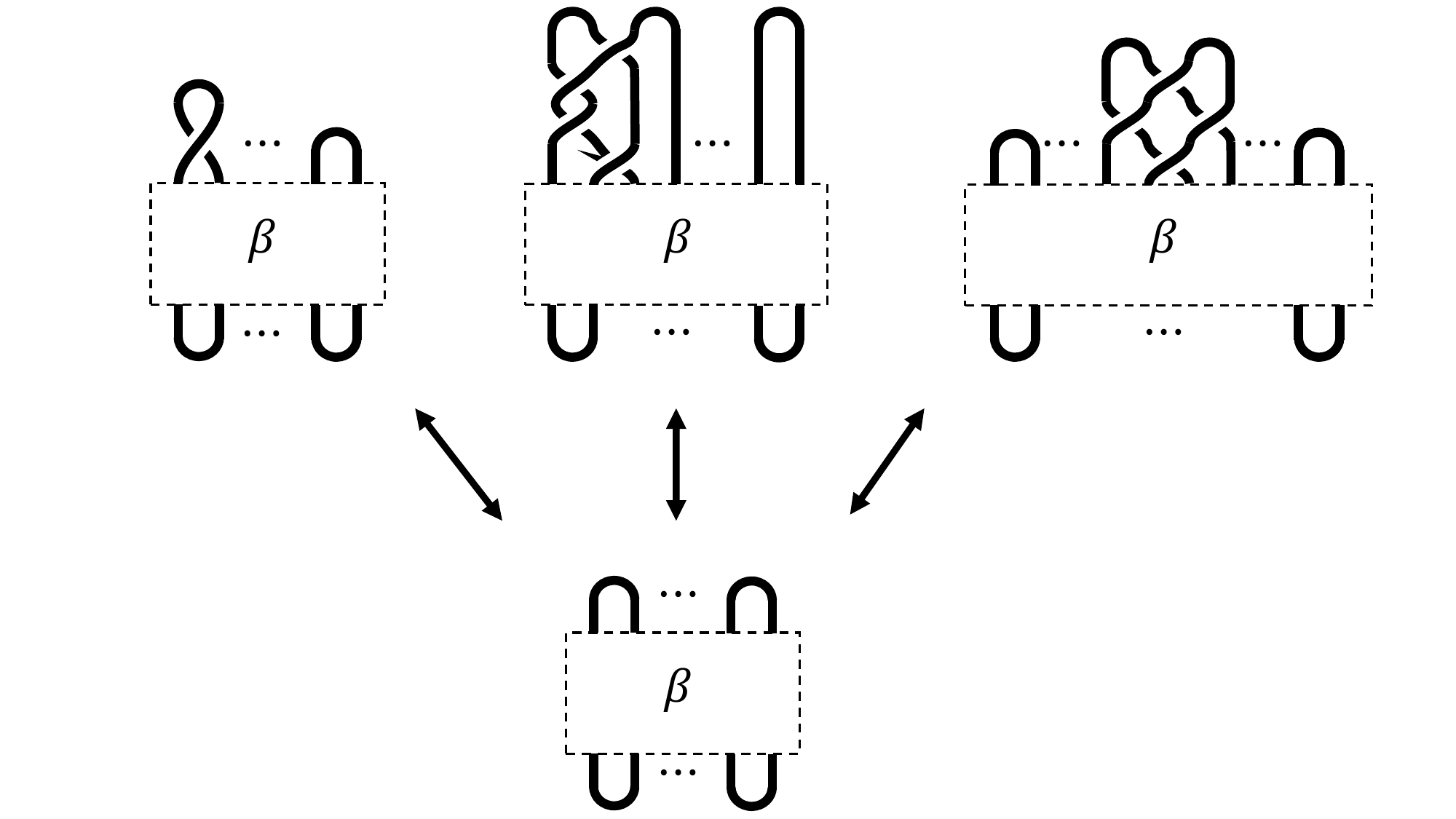}
    \caption{Birman knot and generators of $K_{2n}$.}
    \label{fig:2}
\end{figure}

\end{proposition}

Unlike braids, Birman knots require that identify secant homotopy classes near critical points, i.e.,  $AC \overset{loc}{\sim} XC \overset{loc}{\sim} BC$, where $X$ is a critical point connecting strand $A$ and $B$, while $C$ lies other strand. Thus, for Birman knot, we have %%%%%%%%%%%%%%%%%%%%%% a little modified
$$(i,k)_{s}\sim (i,l)_{s}\sim (j,l)_{s}\sim (j,k)_{s}$$
where $s=-0.25+\epsilon, 1.25-\epsilon$, for a sufficiently small $\epsilon>0$, when two strands $i_t, j_t$ meet at a critical point and so do $k_t,l_t$. The secants homotopic to critical points, also need special treatment. 

Let $\mathcal{E}_ K \subset \mathcal{S}_k$ be the set of secants homotopic to critical points. A secant $(k,l)$ in  $\mathcal{E}_ K$ is represented by $\overrightarrow{e_{j\pm}}$ if it is homotopic to the critical point with number $j$ at top ($+$) or at bottom ($-$) respectively, such that the orientation of $(k,l)$ is from left to right near the critical point. For a secant $(k,l)$ in  $\mathcal{E}_ K$ represented by $\overrightarrow{e_{j\pm}}$, the secant $(l,k)$ will be represented by $\overleftarrow{e_{j\pm}}$.

\begin{definition}\label{def9}
A \emph{secant-quandle} $SQ(K)$ of knot $K=\hat{\beta}_{2n}$ is presented as
$$SQ(K):=\Gamma \langle \mathcal{S} _K\mid  \mathcal{T}_K,\mathcal{I}_K \rangle,$$
where the relations $\mathcal{T}_K$ means the actions induced by any trisecant $(j,k,l)$ as in Definition \ref{def3}; $\mathcal{I}_K$ means two types relations:
\begin{itemize}
    \item for any secant $a\in \mathcal{S}_K$ and $\overrightarrow{e_{j\pm}},\overleftarrow{e_{j\pm}}\in \mathcal{E}_K$, we have
\begin{equation}\label{eq:6}
a\circ \overleftrightarrow{e_{j\pm}} =a.
\end{equation}
\item For a sufficiently small $\epsilon>0$, we have the following identification:
$$(i,k)_{s}\sim (i,l)_{s}\sim (j,l)_{s}\sim (j,k)_{s}$$
where $s=-0.25+\epsilon, 1.25-\epsilon$, for a sufficiently small $\epsilon>0$. when two strands $i_t, j_t$ meet at a critical point and so do $k_t,l_t$.  
\end{itemize}
\end{definition}

\begin{lemma}\label{lemma1}
Let $(j,k,l)$ be a trisecant of  a knot. If $(k,l)$ is local secant homotopic to critical point, then we have $(j,l){_+}=(j,l){_-}$ and $(l,j){_+}=(l,j){_-}$.  
\end{lemma}
\begin{proof} If $k_t,l_t$ is connected with a critical point, then $(j,k)_\pm =(j,l)_\pm$. Since $a\circ a=a$ is an axiom of quandles, from the equation (\ref{eq:1}) we obtain the result $(j,l){_+}=(j,l){_-}$. Since $(k,l), (l,k) \in \mathcal{E}_K$, from the equation (\ref{eq:2}) the equality $(l,j){_+}=(l,j){_-}$ follows.
\end{proof}

\begin{lemma}\label{lemma2}
 Assume that two trisecants $(j^\alpha, k^\alpha, l^\alpha)_{\alpha =0,1} \subset \partial D \subset R_{jl}$, where $D \subset $ is a film-frame contained in the ruled surface $R_{jl}$ consisting of strands $j_{t}$ and $l_{t}$, and $k^{0}_{t_0},k^{1}_{t_1}$ are two consecutive intersection points of strands $k^{0}_{t},k^{1}_{t}$ with $R_{jl}$. If $k^{0}_{t}=k^{1}_{t}$ or the strands $k^{0}_{t},k^{1}_{t}$ are connected to a critical point and the secant $(k^{0}_{t},k^{1}_{t})$ is homotopic to a critical point (see Figure~\ref{fig:3}), then we have $${(j^0,l^0)} {_-}={(j^1,l^1)}{_+}~\text{and}~{(l^0,j^0)} {_-}={(l^1,j^1)}{_+}.$$
\end{lemma}
\begin{proof}
We check whether the action is also compatible with this conclusion. In our convention, the following equalities hold:
\begin{equation}
j^0{_+}=j^1{_-},\quad l^0{_+}=l^1{_-}. \notag
\end{equation}
Note that two trisecants $\{(j^\alpha, k^\alpha, l^\alpha)\}_{\alpha=0,1}$ have different signs, say, $\mathrm{sign}(j^0, k^0, l^0)> 0$ and $\mathrm{sign}(j^1, k^1, l^1)<0$ (or vice versa). From trisecants $\{(j^\alpha, k^\alpha, l^\alpha)\}_{\alpha=0,1}$ we obtain
\begin{eqnarray}
 (j^0,l^0){_+}=(j^0,l^0){_-} \circ (j^0,k^0), \quad
 (l^0,j^0){_+}=(l^0,j^0){_-} / (l^0,k^0), \label{eq:4}\\
 (j^1,l^1){_+}=(j^1,l^1){_-}/ (j^1,k^1), \quad
 (l^1,j^1){_+}=(l^1,j^1){_-}\circ (l^0,k^0),\label{eq:5}
\end{eqnarray}
since $(j^0,l^0)_+=(j^1,l^1)_-$.

If $k^{0}_{t} = k^{1}_{t}$, then $k^0_+ = k^1_-$ and $(j^{0},k^{0}) = (j^{1},k^{1})$. We obtain that
\begin{eqnarray*}
    (j^1,l^1){_+}&=&(j^1,l^1){_-}/\ (j^1,k^1) \\
   &=& ((j^0,l^0){_-}\circ (j^0,k^0))/ (j^1,k^1) \\
   &=& (j^0,l^0){_-}
\end{eqnarray*} 
If the strands $k^{0}_{t},k^{1}_{t}$ are connected to a critical point and the secant $(k^{0}_{t},k^{1}_{t})$ is homotopic to a critical point, then $(j^{0},k^{0}) \sim (j^{1},k^{1})$ and we obtain $(j^1,l^1){_+} = (j^0,l^0){_-}$.

Analogously, we can show that ${(l^0,j^0)} {_-}={(l^1,j^1)}{_+}$.
\end{proof}
\begin{definition}
\label{def5} 
Let $K$ be a Birman knot. $(j,k,l)$ is a \emph{trivial trisecant} of a knot if it belongs to one of the following two types. 
\begin{itemize}
    \item \emph{Type 1.} $(j,k)$ or $(k,l)\in \mathcal{E}_K$. 
    \item \emph{Type 2.}   Two trisecants $(j^\alpha, k^\alpha, l^\alpha)\subset \partial D$, $\alpha=0,1$, satisfy the conditions of Lemma \ref{lemma2} see Figure \ref{fig:3}. 
   
\end{itemize}

\end{definition}
\begin{figure}
    \centering
    \includegraphics[width=1\linewidth]{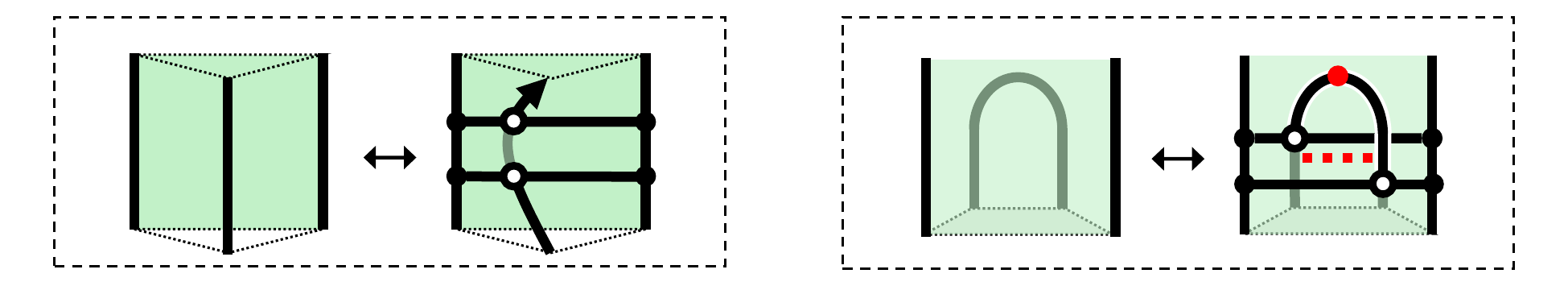}
    \caption{Type 2}
    \label{fig:3}
\end{figure}

\begin{remark}
    The set of all Type 1 trivial trisecants is denoted by $\mathcal{T}_1$. If a trisecant $(j,k,l)\in \mathcal{T}_1$, then the trisecant $(j,k,l)$ provides the relations $(j,l){_+}= (j,l){_-}$ and $(l,j){_+}= (l,j){_-}$. The set of all Type 2 trivial trisecants is denoted $\mathcal{T}_2$.
\end{remark}
%%%%%%%%%%%%%%%%%%%%%%%%%%%%%%%%%%%%%%%%%%%%%%%%%%% Section 3
\section{The proof of the main theorems}
\begin{lemma}
\label{lemma3}
%\rm{(Stability under trivial trisecant additions)}
Let $M$ and $M'$ be two braid (or knot) diagrams that coincide except the trisecant 
\begin{itemize}
    \item If $M=K$, $\mathcal{T}_{M'}=\mathcal{T}_{M}\cup \mathcal{T}'$, where $\mathcal{T}'=\{(j,k,l)\}$ is Type 1; 
    \item If $M=\beta$ or $K$, $\mathcal{T}'=\{(j^\alpha, k^\alpha, l^\alpha)\}_{\alpha=0,1}$ is Type 2.
\end{itemize}
Then the associated secant-quandles are isomorphic:
$$SQ(M)\cong SQ(M').$$
\end{lemma}
\begin{proof}
We construct explicit isomorphisms. 

\textbf{Case 1:} Assume that $\mathcal{T}'=\{(j,k,l)\}$ is Type 1 trivial.
By Definition \ref{def5}, the secant $(j,k)$ or $(k,l)$ is trivial, hence $(j,l)_{+}=(j,l)_{-}$ and $(l,j)_{+}=(l,j)_{-}$. Denote the associated secant in $SQ(M)$ by $(j,l)$. Define a map $f: SQ(M') \to SQ(M)$ on generators by $f((j,l)_\pm ) = (j,l)$, $f((l,j)_\pm ) = (l,j)$ and $f(a)=a$ for all other $a\in\mathcal{S}_{M'}$. The Type 1 trisecant introduces no new independent generator, so $f$ is a bijective, hence an isomorphism. %%%comment:here we need to consider two secants (j,l), (l,j).

\textbf{Case 2:} Assume that $\mathcal{T}'=\{(j^\alpha, k^\alpha, l^\alpha)\}_{\alpha=0,1}$ is Type 2 trivial.
For such a pair, Lemma \ref{lemma2} gives ${(j^0,l^0)} {_-}={(j^1,l^1)}{_+}$ and ${(l^0,j^0)} {_-}={(l^1,j^1)}{_+}$ in $SQ(M')$. Let $(j,l) \in SQ(M)$ be the associated secant with $(j^0,l^0)_{-}$ and $(j^1,l^1)_{+}$. Define $g: SQ(M') \to SQ(M)$ on generators by $g((j^0,l^0)_{-}) = g((j^1,l^1)_{+}) = (j,l)$ and $g(a)=a$ for all remaining $a\in\mathcal{S}_{M'}$.
Then
$$g((j^0,l^0)_{+}) = g((j^0,l^0)_{-} \star (j^0,k^0)) = (j,l) \star (j,k),$$
$$g((j^1,l^1)_{-}) = g((j^0,l^0)_{+} \star (j^1,k^1)) = (j,l) \star (j,k),$$
where $\star = \circ,/$. Similarly, $g((l^0,j^0)_+)=g((l^1,j^1)_-)$. The map is bijective since the ``new'' generators $(j^0,l^0)_{+}$ and $(j^1,l^1)_{-}$ of $SQ(M')$ are uniquely presented by others. Hence $g$ is an isomorphism. 
\end{proof}

\begin{theorem}\label{thm1}\rm{(Main theorem for braids)}
The secant-quandle $SQ(\beta)$ is a braid invariant, i.e., if two braids $\beta$ and $\beta'$ are equivalent, then $SQ(\beta)\cong SQ(\beta')$. 
\end{theorem}

\begin{proof}
There are three \emph{trisecant moves} of braids: $\Lambda_j (j=1,2,3)$ as illustrated in Figure \ref{fig:4}, \ref{fig:5}, \ref{fig:6} and \ref{fig:7}. From a standard codimension argument it follows that the set of moves $\{\Lambda_j\}_{j=1}^3$ is complete for the isotopy of braid in $\mathbb{C}\times \mathbb{I}$.

A one-parameter family of braids (an isotopy) encounters only finitely many non-generic events. The non-generic events can be divided into following three cases: 
\begin{itemize}
    \item \textbf{The tangency event}: two adjacent trisecants coincide, i.e., strand $k_t$ is tangent to $R_{jl}$; 
    \item \textbf{The 1-intersection event}: two trisecants intersect at most one point; 
    \item \textbf{The Collinear event}: more than 2 trisecants are collinear, e.g. quadrisecants. 
\end{itemize}
The first two events correspond to the moves $\Lambda_1$ and $\Lambda_2$. It remains to show that $\Lambda_3$ is sufficient to handle the collinear case. Any potential event involving the collinearity of five or more points resides in a stratum of codimension $\geq 3$ and is avoided by a generic isotopy path. This classification and completeness result is strongly supported by the comprehensive framework developed by Fiedler and Kurlin for links in a solid torus [\cite{Fiedler2}, Theorem 1.4]. 

%In their theory of trace graphs, the complete set of local moves generating all isotopies. Since we only consider braid moves rather than knot moves, {\color{red} the set ultimately filtered down from these 11 moves is $\Lambda_j (j=1,2,3)$.}

%%%comment: which 11 moves? Here I cannot understand what "filtered down" means neither.

%%reply: May I delete the last sentence? 

\begin{figure}%[H]  % [H] 强制图片位置不浮动
    \centering
    \includegraphics[width=1\linewidth]{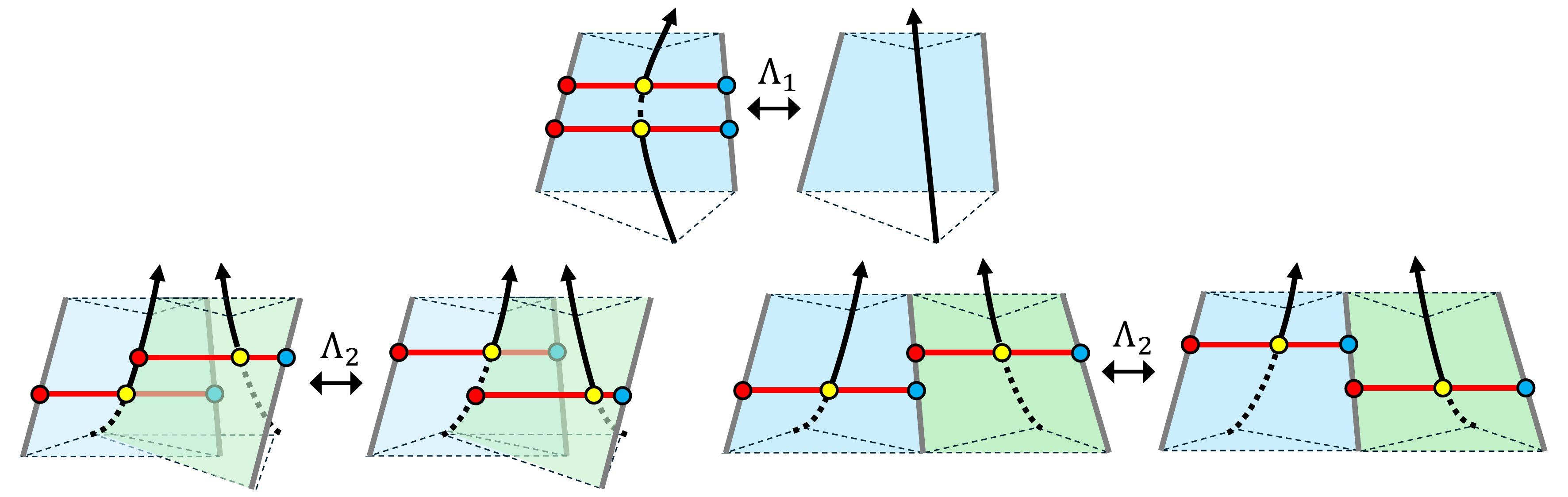}
    \caption{Trisecant isotopy moves.}
    \label{fig:4}
\end{figure}

\begin{figure}%[H]  % [H] 强制图片位置不浮动
    \centering
    \includegraphics[width=1\linewidth]{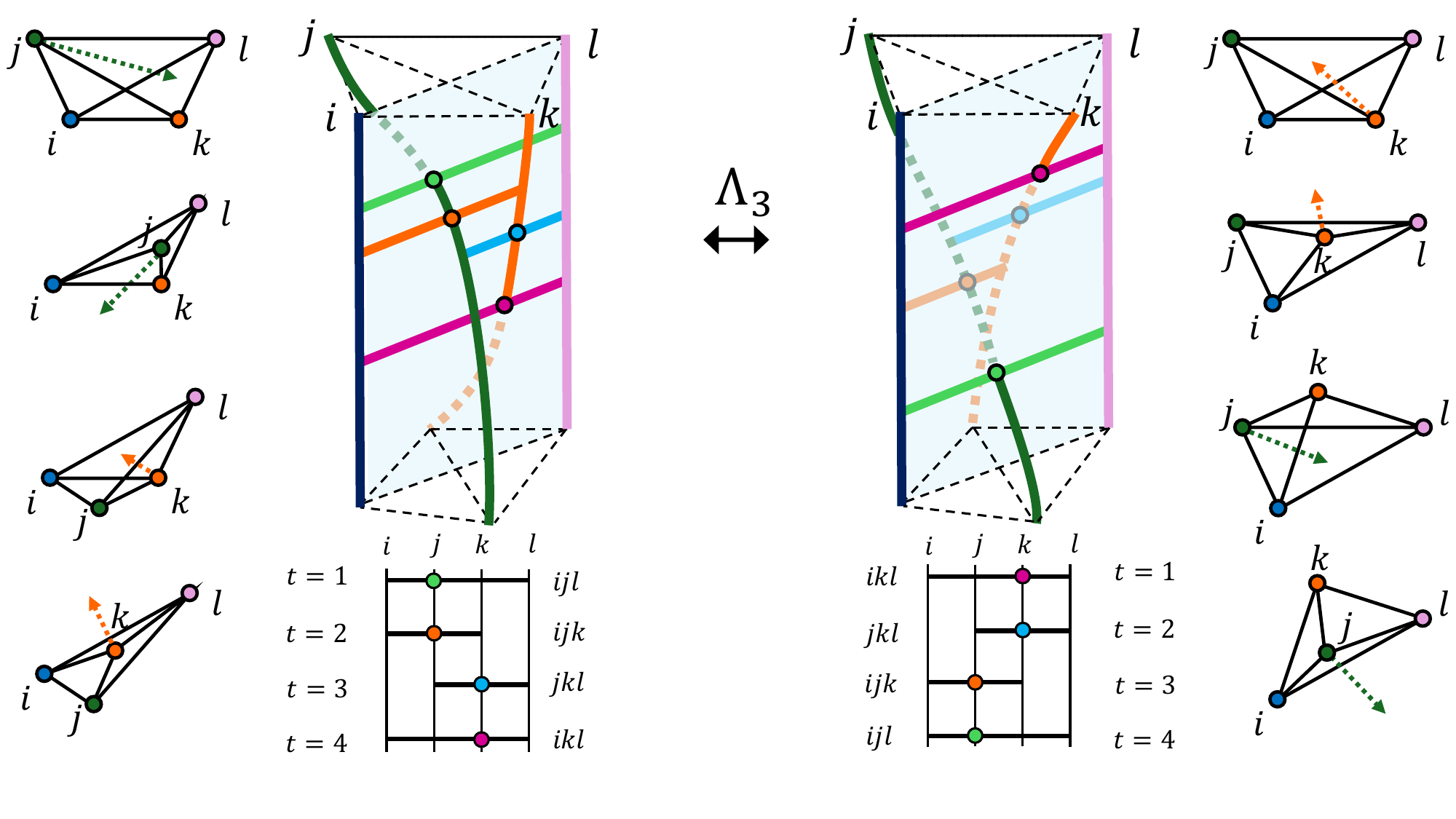}
    \caption{$\Lambda_3$: Case 1.}
    \label{fig:5}
\end{figure}

\begin{figure}%[H]  % [H] 强制图片位置不浮动
    \centering
    \includegraphics[width=1\linewidth]{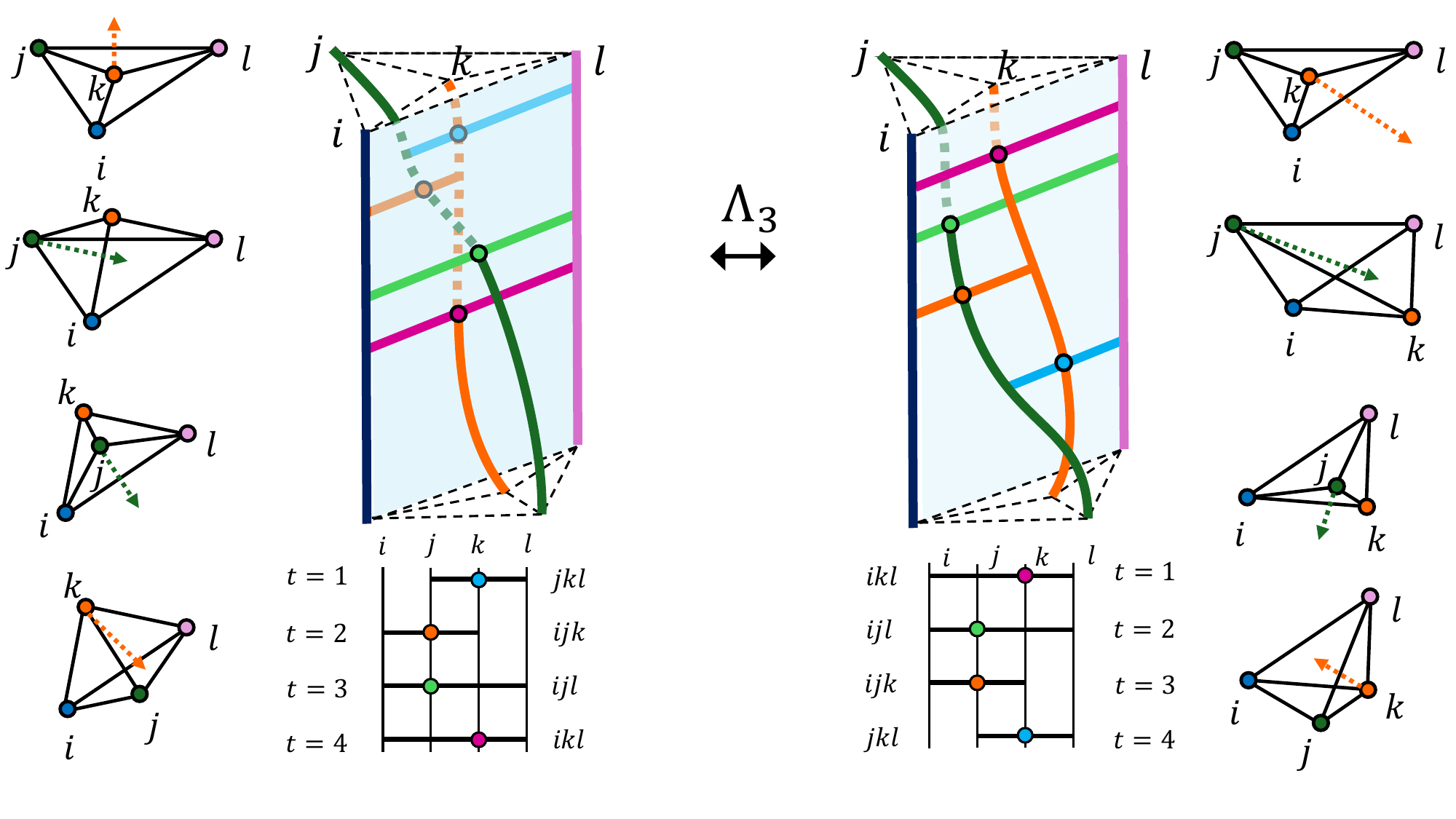}
    \caption{$\Lambda_3$: Case 2.}
    \label{fig:6}
\end{figure}

\begin{figure}%[H]  % [H] 强制图片位置不浮动
    \centering
    \includegraphics[width=1\linewidth]{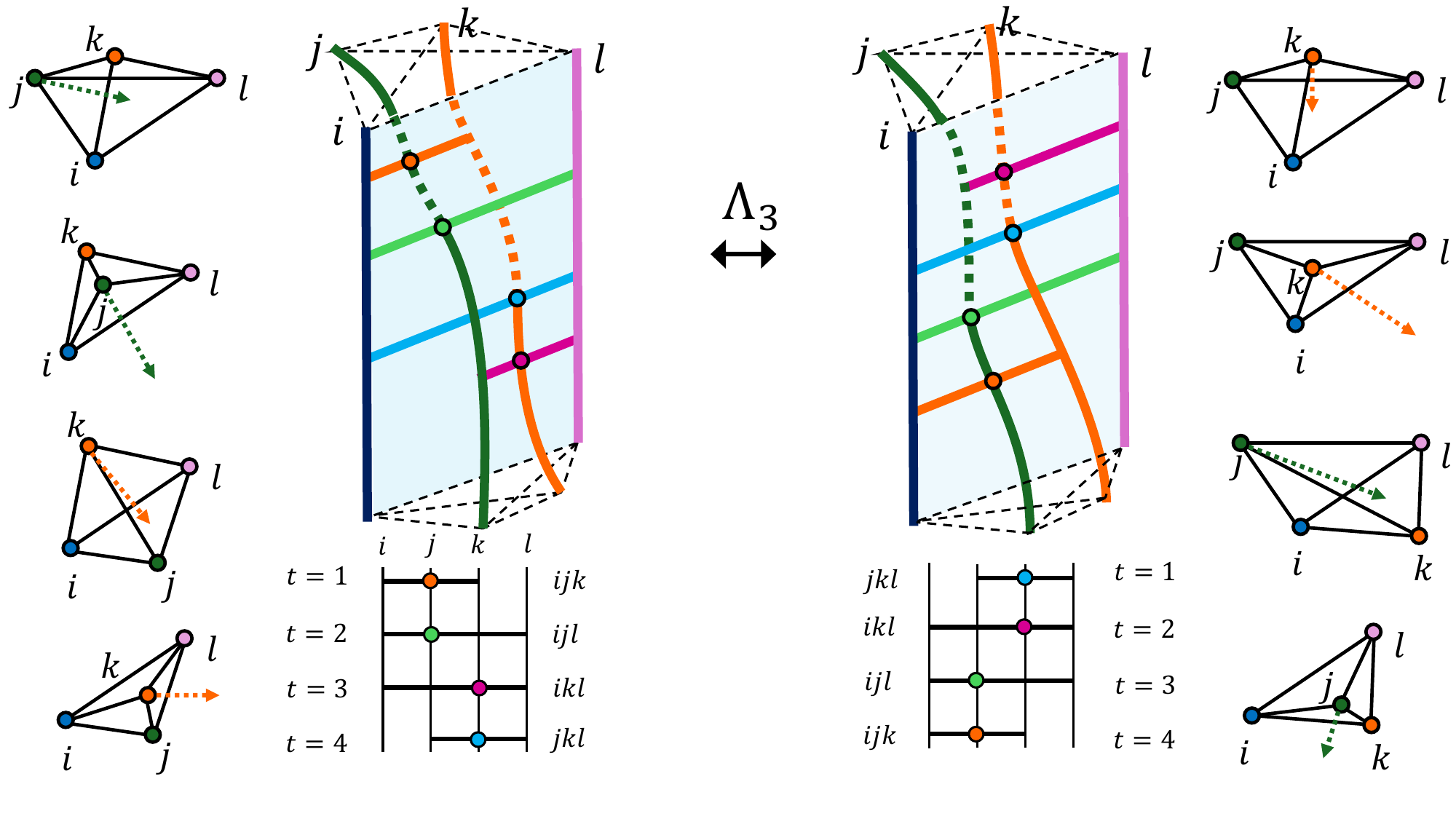}
    \caption{$\Lambda_3$: Case 3.}
    \label{fig:7}
\end{figure}

Now let us show that $SQ(\beta)$ is invariant under moves $\Lambda_j$. By Lemma \ref{lemma3}, $\Lambda_1$ does not influence on $SQ(\beta)$, since it makes Type 2 trisecants; $\Lambda_2$ induce the identity map of $SQ(\beta)$ to itself, since the generators $\mathcal{S}_\beta$ and relations $\mathcal{T}_\beta$ remain. Thus, we only need to check $\Lambda_3$, i.e., a ruled surface $R_{il}$ intersects two strands $j$ and $k$ locally involving those trisecants $(i,j,k), (i,j,l), (i,k,l),(j,k,l)$ and their opposite oriented trisecants. We just need to check three cases by symmetry: 

\textbf{Case 1.} $j_0$ and $k_0$ lie opposite sides of $R_{il}$, see Figure \ref{fig:5};

\textbf{Case 2.} $j_0$ and $k_0$ lie the same sides of $R_{il}$, and $k_0$ enclosed by the triangle $\triangle i_0j_0l_0$, see Figure \ref{fig:5};

\textbf{Case 3.} $j_0$ and $k_0$ lie the same sides of $R_{il}$, but $k_0$ is outside of the triangle $\triangle i_0j_0l_0$, see Figure \ref{fig:7}.

Let us check Case 1, and Case 2 and 3 are left to the readers. Obviously, secants $(i,j)$, $(j,k)$ and $(k,l)$ remain if $t\in [1,4]$, since $R_{ij}$, $R_{jk}$ and $R_{kl}$ does not intersect any other strands, whether using $\Lambda_3$ move. Note that trisecants $(i,j,l)$ and $(i,j,k)$ have the same sign and so do $(j,k,l)$ and $(i,k,l)$.We only prove the ``right secant'' $(x,y)$ case here; the ``left secant'' $(y,x)$ case is similar.  For the left of Figure \ref{fig:5}, according these four trisecants we have
$$\begin{array}{rrrl}
(i,j,l)_{0.5}:& (i,l)_1 &=&(i,l)_0\star (i,j)\\ 
(i,j,k)_{1.5}:&(i,k)_2 &=&(i,k)_1\star (i,j) \\
(j,k,l)_{2.5}:&(j,l)_3 &=&(j,l)_2* (j,k)\\
(i,k,l)_{3.5}:&(i,l)_4 &=&(i,l)_3* (i,k)_3
\end{array}$$ where $\star,*=\circ,/$. Since $(i,k){_+}:=(i,k)_0=(i,k)_1\neq (i,k)_2=(i,k)_3=(i,k)_4=:(i,k){_-}$, $(i,l){_+}:=(i,l)_0\neq (i,l)_1=(i,l)_2=(i,l)_3\neq (i,l)_4=:(i,l){_-}$ and $(j,l){_+}:=(j,l)_0=(j,l)_1=(j,l)_2\neq (j,l)_3= (j,l)_4:=(j,l){_-}$,
$$\begin{array}{lll}\label{eq:7}
 (i,l){_-}=(i,l)_4 &=& (i,l)_1*(i,k)_2 \\
    &=&((i,l)_0\star (i,j) ) * ((i,k)_1\star (i,j))\\
    &=&((i,l)_0\star (i,j) ) * ((i,k)_0\star (i,j))\\
    &=&((i,l)_0*(i,k)_0 ) \star (i,j)\\
    &=&((i,l){_+}*(i,k){_+} ) \star (i,j)\\
\end{array}$$
$$\begin{array}{lll}\label{eq:8}
 (j,l){_-}=(j,l)_4 &=& (j,l)_2*(j,k) \\
    &=& (j,l)_0*(j,k)\\
    &=& (j,l){_+}*(j,k).\\
\end{array}$$
For the right part of Figure \ref{fig:5}, according these four trisecants we have
$$\begin{array}{rrrr}
(i,k,l)'_{0.5}:&(i,l)'_1&=&(i,l)'_0* (i,k)'_0\\
(j,k,l)'_{1.5}:&(j,l)'_2&=&(j,l)'_1* (j,k)'\\
(i,j,k)'_{2.5}:&(i,k)'_3&=&(i,k)'_2\star (i,j)' \\
(i,j,l)'_{3.5}:&(i,l)'_4&=&(i,l)'_3\star (i,j)' 
\end{array}$$
Since $(i,k)'{_+}:=(i,k)'_0=(i,k)'_1=(i,k)'_2\neq (i,k)'_3=(i,k)'_4=:(i,k)'{_-}$,  $(i,l)'{_+}:=(i,l)_0\neq (i,l)'_1=(i,l)'_2=(i,l)'_3\neq (i,l)'_4=:(i,l)'{_-}$ and $(j,l)'{_+}:=(j,l)'_0=(j,l)'_1\neq (j,l)'_2=(j,l)'_3=(j,l)'_4=:(j,l)'{_-}$,
$$\begin{array}{lll}\label{eq:9}
 (i,l)'{_-}= (i,l)'_4 &=& (i,l)'_1\star (i,j)' \\
    &=&((i,l)'_0*(i,k)'_0)\star (i,j)'\\
    &=&((i,l)'{_+}*(i,k)'{_+})\star (i,j)'
\end{array}$$
$$\begin{array}{lll}\label{eq:10}
 (j,l)'{_-}=(j,l)'_4 &=& (j,l)'_1*(j,k)' \\
    &=& (j,l)'_0*(j,k)'\\
    &=& (j,l)'{_+}*(j,k)'
\end{array}$$
Consider the following morphism
$$\begin{array}{rrcl}
f:   &  (i,j)&\mapsto& (i,j)'\\
     & (j,k)&\mapsto& (j,k)'\\
     & (k,l)&\mapsto& (k,l)'\\
     & (i,k)_\pm &\mapsto& (i,k)'_\pm\\
     & (i,l)_\pm &\mapsto& (i,l)'_\pm\\
     & (j,l)_\pm &\mapsto& (j,l)'_\pm
\end{array}$$
It is clear that $f:SQ(\beta)\rightarrow SQ(\beta')$ where $\beta'=\Lambda_3(\beta)$ is an isomorphism.  Thus, $SQ(\beta)$ is an isotopy invariant of $\beta$. 
\end{proof}
%{\color{red} We will use notations $(i,j,k)$ for trisecant when $\text{sign}(i,j,k)>0$ is positive. }%%%comment We use this notation like this from the beginning.

%%reply: this is written by you, so now we are unify all notations, i think we can delete it right?
\begin{lemma}\label{lemma4}
    Let $\beta \in B_{2n}$ and $\beta'=\beta\sigma_{2n}\in B_{2n+2}$. Let $K$ and $K'$ be plat closures of $\beta$ and $\beta'$. Then $SQ(K)$ and $SQ(K')$ are isomorphic.
\end{lemma}
\begin{proof}
    Assume that $\beta$ is placed in $\mathbb{R}^{2}\times [0,1]$ and $\sigma_{2n}$ for $\beta'$ is placed in $\mathbb{R}^{2}\times[1,1.25]$. Note that \begin{small}$$\mathcal{S'} = \mathcal{S} \cup \{\text{secants}~(i,2n+1),(i,2n+2), (2n+1,i),(2n+2,i), (2n+1,2n+2), (2n+2,2n+1)\}.$$
    \end{small}
    Let $\mathcal{T}_{K}$ is the set of actions induced from trisecants of $K$. Then
    \begin{eqnarray*}
       \mathcal{T}_{K'} = \mathcal{T}_{K}&\cup& \{\text{trisecants in [0,1] containing points}~2n+1,2n+2\}\\ &\cup& \{\text{trisecants in}~ [1,1.25]\}. 
    \end{eqnarray*}
    We will show that relations in $\mathcal{T}_{K'} -\mathcal{T}_{K}$ do not reduce $\mathcal{S}$ and generators in $\mathcal{S}'-\mathcal{S}$ can be presented by $\mathcal{S}$.
    Let $$0=t_{0}<t_{1}<t_{2}<\dots<t_{a}<1<t_{a+1}<\dots<t_{b}$$ be the heights where trisecants appear for $\beta'$. Let us denote by $(i,j)_{t_{s}}$ be a (homotopy class of) secant from $i$ to $j$ in $[t_{s},t_{s+1}]$. Notice that if a trisecant $(i,j,k)$ with a positive sign appears at $t_{s}$, then 
\begin{eqnarray*}
        (i,k)_{t_{s+1}} &=& (i,k)_{t_{s}} \circ (i,j)_{t_{s}},\\
        (k,i)_{t_{s+1}} &=& (k,i)_{t_{s}} / (k,j)_{t_{s}},\\
        (l,m)_{t_{s}}&=&(l,m)_{t_{s+1}},~\text{otherwise.}
    \end{eqnarray*}
    In this proof, we assume that a trisecant $(i,j,k)$ has a positive sign. If $(i,j,k)$ has a negative sign, then we will denote it by $(k,j,i)$ instead.
    
    The proof consists of three parts: 
    
    \textbf{Part 1.} We consider secants and trisecants in $[0,1]$. Note that trisecants in $[0,1]$ containing points $2n+1,2n+2$ have one of the forms of $(i,j,2n+1),(i,j,2n+2)$ for $1\leq i,j\leq 2n$. 
    It follows that $(2n+1,2n+2)_{t_{s}} = (2n+1,2n+2)_{t_{0}}=\overrightarrow{e_{(n+1)+}}$ and $(2n+2,2n+1)_{t_{s}} = (2n+2,2n+1)_{t_{0}}=\overleftarrow{e_{(n+1)+}}$ for $0\leq t_s \leq 1$. 
    
Without loss of generality, we assume that if $(i,j,2n+1)$ (or $(2n+1,j,i)$) appears at $t_{s}$, then $(i,j,2n+2)$ ($(2n+2,j,i)$ resp.) appears at $t_{s\pm 1}$.
    
    {\bf Claim.} If a pair of trisecants $(i,j,2n+1),(i,j,2n+2)$ or $(2n+1,j,i),(2n+2,j,i)$ appear at $t_{s}$ and $t_{s+1}$, then $(k,l)_{t_{s-1}} = (k,l)_{t_{s+1}},$ for $k,l \neq 2n+1,2n+2$ and $(k,2n+1)_{t_{s-1}}=(k,2n+2)_{t_{s+1}}$. That is, trisecants $(i,j,2n+1),(i,j,2n+2)$ do not affect on secants $(k,l)$ for $k,l \neq 2n+1,2n+2$.

    Since the points $2n+1$ and $2n+2$ meets at the extreme point, $(i,2n+1)_{t_{0}} = (i,2n+2)_{t_{0}}$.
    Let $(i,j,2n+1),(i,j,2n+2)$ be the first pair of trisecants containing $2n+1$ and $2n+2$ which appear at $t_{s}$ and $t_{s+1}$. We obtain equalities 
    \begin{eqnarray*}
        (i,2n+1)_{t_{s}} &=& (i,2n+1)_{t_{s-1}}\circ (i,j)_{t_{s-1}}\\
        (2n+1,i)_{t_{s}} &=& (2n+1,i)_{t_{s-1}}\circ (2n+1,j)_{t_{s-1}}\\(i,2n+2)_{t_{s+1}} &=& (i,2n+2)_{t_{s}}\circ (i,j)_{t_{s}}\\
        (2n+2,i)_{t_{s+1}} &=& (2n+2,i)_{t_{s}}\circ (2n+2,j)_{t_{s}}
    \end{eqnarray*}
    It is clear that $(i,j)_{t_{s-1}} = (i,j)_{t_{s}}=(i,j)_{t_{s+1}}$ and $(i,2n+2)_{t_{s}} = (i,2n+2)_{t_{s-1}}$, $(2n+2,i)_{t_{s}} = (2n+2,i)_{t_{s-1}}$,$ (i,2n+1)_{t_{s+1}} = (i,2n+1)_{t_{s}}$, $(2n+1,i)_{t_{s+1}} = (2n+1,i)_{t_{s}}$.
    By assumption, $$(i,2n+1)_{t_{s-1}} = (i,2n+1)_{t_{0}}= (i,2n+2)_{t_{0}}= (i,2n+2)_{t_{s-1}},$$
    and
    $$(2n+1,i)_{t_{s-1}} =(2n+1,i)_{t_{0}} = (2n+2,i)_{t_{0}} =(2n+2,i)_{t_{s-1}}.$$
    Therefore we obtain that 
    \begin{eqnarray*}
        (i,2n+1)_{t_{s}} &=& (i,2n+1)_{t_{s-1}}\circ (i,j)_{t_{s-1}}\\
        &=& (i,2n+2)_{t_{s-1}}\circ (i,j)_{t_{s-1}} \\
        &=& (i,2n+2)_{t_{s}}\circ (i,j)_{t_{s}} \\
        &=&(i,2n+2)_{t_{s+1}},
    \end{eqnarray*}
    and it follows that
    $$(i,2n+1)_{t_{s+1}} = (i,2n+1)_{t_{s}} = (i,2n+2)_{t_{s+1}}.$$
    Analogously,
     \begin{eqnarray*}
        (2n+1,i)_{t_{s}} &=& (2n+1,i)_{t_{s-1}}\circ (2n+1,j)_{t_{s-1}}\\
        &=& (2n+2,i)_{t_{s-1}}\circ (2n+2,j)_{t_{s-1}}\\
        &=& (2n+2,i)_{t_{s}}\circ (2n+2,j)_{t_{s}}\\
        &=&(2n+2,i)_{t_{s+1}},
    \end{eqnarray*}
    and it follows that 
    $$(2n+1,i)_{t_{s+1}}= (2n+1,i)_{t_{s}}=(2n+2,i)_{t_{s+1}}.$$
    Inductively, we can prove our claim for all pairs of trisecants containing $2n+1, 2n+2$ appearing consecutively.

    \textbf{Part 2.} Let us consider secants and trisecants in $[1,1.25]$. At the time $t=1$, all points are placed in the initial places. At this time, points are numbered by $\sigma^{-1}(1),\dots, \sigma^{-1}(2n), 2n+1,2n+2$ from the left to right, where $\sigma$ is the image of $\beta$ by the natural map from $B_{2n}$ to the permutation group $\Sigma_{2n}$. From Part 1 we obtained that $$(\sigma^{-1}(s),2n+1)_{t_{a}}= (\sigma^{-1}(s),2n+2)_{t_{a}}$$
    and
    $$(2n+1,\sigma^{-1}(s))_{t_{a}}= (2n+2,\sigma^{-1}(s))_{t_{a}}.$$
    In $[1,1.25]$ we can describe all trisecants as below:
    \begin{align*}
      &(\sigma^{-1}(2n),2n+1,2n+2),(\sigma^{-1}(2n),2n+1,\sigma^{-1}(1)),(\sigma^{-1}(2n),2n+1,\sigma^{-1}(2)), \dots \\
      &\dots,(\sigma^{-1}(2n),2n+1,\sigma^{-1}(2n-2)),  
    \end{align*}
    and they appear at $t_{a+1},\dots, t_{b} = t_{a+2n}.$
    We will show that $$(\sigma^{-1}(l),\sigma^{-1}(m))_{t_{a+s}} = (\sigma^{-1}(l),\sigma^{-1}(m))_{t_{a+s+1}},$$
    where $l,m\neq 2n+1,2n+2$, and
$$(\sigma^{-1}(l),2n+1)_{t_{a+s}} = (\sigma^{-1}(l),2n+2)_{t_{a+s}}.$$
    First from the tresecant $(\sigma^{-1}(2n),2n+1,2n+2)$ at $t_{a+1}$ we obtain actions
    \begin{eqnarray*}
        (\sigma^{-1}(2n),2n+2)_{t_{a+1}} &=& (\sigma^{-1}(2n),2n+2)_{t_{a}} \circ (\sigma^{-1}(2n),2n+1)_{t_{a}}\\
        (2n+2,\sigma^{-1}(2n),)_{t_{a+1}} &=& (2n+2,\sigma^{-1}(2n))_{t_{a}} \circ (2n+2,2n+1)_{t_{a}},
    \end{eqnarray*}
    and $(i,j)_{t_{a}} =    (i,j)_{t_{a+1}}$, otherwise.
    Since $(\sigma^{-1}(2n),2n+1)_{t_{a}}= (\sigma^{-1}(2n),2n+2)_{t_{a}}$, we obtain $$(\sigma^{-1}(2n),2n+2)_{t_{a+1}} = (\sigma^{-1}(2n),2n+2)_{t_{a}}=(\sigma^{-1}(2n),2n+1)_{t_{a}}=(\sigma^{-1}(2n),2n+1)_{t_{a+1}}.$$
    Since $(2n+2,2n+1)_{t_{a}} =\overleftarrow{e_{(n+1)+}} $, we obtain that $$(2n+2,\sigma^{-1}(2n))_{t_{a+1}} = (2n+2,\sigma^{-1}(2n))_{t_{a}} =(2n+1,\sigma^{-1}(2n))_{t_{a}} =(2n+1,\sigma^{-1}(2n))_{t_{a+1}}.$$
    
    For each $t=t_{a+1+k}$, the trisecant $(\sigma^{-1}(2n),2n+1,\sigma^{-1}(k))$ appears and it gives equalities:
    \begin{eqnarray*}
        (\sigma^{-1}(2n),\sigma^{-1}(k))_{t_{a+1+k}} &=& (\sigma^{-1}(2n),\sigma^{-1}(k))_{t_{a+k}} \circ (\sigma^{-1}(2n),2n+1)_{t_{a+k}}\\
        (\sigma^{-1}(k),\sigma^{-1}(2n))_{t_{a+1+k}} &=& (\sigma^{-1}(k),\sigma^{-1}(2n))_{t_{a+k}} \circ (\sigma^{-1}(k),2n+1)_{t_{a+k}},
    \end{eqnarray*}
    and $(i,j)_{t_{a+k}} =    (i,j)_{t_{a+1+k}}$ otherwise.

    Since $\sigma^{-1}(2n)$ and $2n+2$ meet at an extreme point in $t>t_{a+2n}$ $$(\sigma^{-1}(2n),2n+2)_{t_{a+s}}= (\sigma^{-1}(2n),2n+2)_{t_{a+2n}}= \overrightarrow{e_{(n+1)-}} $$
    and $$(2n+2,\sigma^{-1}(2n))_{t_{a+s}}= (2n+2,\sigma^{-1}(2n))_{t_{a+2n}}= \overleftarrow{e_{(n+1)-}} ,$$
    for $s = 1,\dots, 2n-1$. It follows that $(\sigma^{-1}(2n),\sigma^{-1}(k))_{t_{a+1+k}} = (\sigma^{-1}(2n),\sigma^{-1}(k))_{t_{a+k}}$ for each $k$. 
    Note that $$(\sigma^{-1}(k),2n+2)_{t_{a+1+s}}=(\sigma^{-1}(k),2n+2)_{t_{a}}=(\sigma^{-1}(k),2n+1)_{t_{a}} = (\sigma^{-1}(k),2n+1)_{t_{a+1+s}},$$
    for $s=1,\dots,2n-1$. Moreover, 
    \begin{eqnarray*}
        (\sigma^{-1}(k),\sigma^{-1}(2n))_{t_{a+1+k}} = (\sigma^{-1}(k),\sigma^{-1}(2n))_{t_{a+1+k+1}} =\dots&=& (\sigma^{-1}(k),\sigma^{-1}(2n))_{t_{a+2n}} \\&=&(\sigma^{-1}(k),2n+2)_{t_{a+2n}}. 
    \end{eqnarray*}
    From the above observations, we obtain that 
    \begin{eqnarray*}
        (\sigma^{-1}(k),2n+2)_{t_{a+2n}} &=&(\sigma^{-1}(k),\sigma^{-1}(2n))_{t_{a+1+k}} \\&=& (\sigma^{-1}(k),\sigma^{-1}(2n))_{t_{a+k}} \circ (\sigma^{-1}(k),2n+1)_{t_{a+k}}\\
        &=& (\sigma^{-1}(k),\sigma^{-1}(2n))_{t_{a+k}} \circ (\sigma^{-1}(k),2n+2)_{t_{a+2n}},
    \end{eqnarray*}
    or equivalently,
    \begin{eqnarray*}
        (\sigma^{-1}(k),\sigma^{-1}(2n))_{t_{a+k}} &=& (\sigma^{-1}(k),2n+2)_{t_{a+2n}}/(\sigma^{-1}(k),2n+2)_{t_{a+2n}} \\
        &=& (\sigma^{-1}(k),2n+2)_{t_{a+2n}} = (\sigma^{-1}(k),\sigma^{-1}(2n))_{t_{a+1+k}}
    \end{eqnarray*}

    \textbf{Part 3.}
    In Parts 1 and 2, we have shown that 
    \begin{itemize}
        \item $(k,l)$ remains under actions from trisecants containing $2n+1$ or $2n+2$ for $1\leq k,l\leq2n$.
        \item $(i,2n+1)_{t_{s}} = (i,2n+2)_{t_{s}}$, for $s=0,\dots, a$.
        \item $(\sigma^{-1}(k),2n+1)_{t_{a+s}} = (\sigma^{-1}(k),2n+1)_{t_{a+s+1}}$ for $s=0,\dots 2n-1$.
        \item $(\sigma^{-1}(k),2n+2)_{t_{a+s}} = (\sigma^{-1}(k),2n+2)_{t_{a+s+1}}$ for $s=0,\dots 2n-1$.
        \item
        $(2n+1,\sigma^{-1}(k))_{t_{a+s}} = (2n+1,\sigma^{-1}(k))_{t_{a+s+1}}$ for $s=0,\dots 2n-1$.
         \item
        $(2n+2,\sigma^{-1}(k))_{t_{a+s}} = (2n+2,\sigma^{-1}(k))_{t_{a+s+1}}$ for $s=0,\dots 2n-1$.
        \item $(2n+1,2n+2)_{t_{s}} = \overrightarrow{e_{(n+1)+}} , (2n+2,2n+1)_{t_{s}} =\overleftarrow{e_{(n+1)+}} ,
        (\sigma^{-1}(2n),2n+1)_{t_{a+k}} =(\sigma^{-1}(2n),2n+2)_{t_{a+k}} = \overrightarrow{e_{(n+1)-}} ,  (2n+1, \sigma^{-1}(2n))_{t_{a+k}} = (2n+2, \sigma^{-1}(2n))_{t_{a+k}} = \overleftarrow{e_{(n+1)-}} $
    \end{itemize}
    and they are all possible equalities coming from $\mathcal{T}_{K'}-\mathcal{T}_{K}$.

    To finish the proof, we need to show that all secants $(s,2n+1)$, $(s,2n+2)$, $(2n+1,s)$ and $(2n+2,s)$ can be presented by secants $(k,l)$ for $k,l\neq 2n+1,2n+2$ for any time $t$. It is sufficient to show that $(\sigma^{-1}(k),2n+1)_{t_{a+s}}$, $(\sigma^{-1}(k),2n+2)_{t_{a+s}}$, $(2n+1,\sigma^{-1}(k))_{t_{a+s}}$, $(2n+2,\sigma^{-1}(k))_{t_{a+s}}$ can be presented by $(k,l)$ for some $s=0, \dots, 2n$.\footnote{Notice that if $(\sigma^{-1}(k),2n+2)_{t_{a+s}}$ can be presented by $(k,l)$ for some $s$, then so does $(\sigma^{-1}(k),2n+2)_{t_{a+s}}$ for all $s =0,\dots, 2n$. For $s =-1, \dots, -a$, since $(\sigma^{-1}(k),2n+2)_{t_{a+s}}= (\sigma^{-1}(k),2n+2)_{t_{a+s+1}}*(\sigma^{-1}(k),i_{j})$ for some $i_{j}$, where $*=\circ,/$, $(\sigma^{-1}(k),2n+2)_{t_{a+s}}$ can be presented by $(k,l)$ for all  $s =-1, \dots, -a$.}
    We already showed that $(\sigma^{-1}(k),\sigma^{-1}(2n))_{t_{a+k}} = (\sigma^{-1}(k),2n+2)_{t_{a+k}}$. It is clear that $$(2n+2,\sigma^{-1}(k))_{t_{a+k}} = (\sigma^{-1}(2n),\sigma^{-1}(k))_{t_{a+k}}.$$
    Finally, since $\sigma^{-1}(2n-1)$ and $2n+1$ meet at an extreme point, thus
    $$(2n+1,\sigma^{-1}(k))_{t_{a+2n}} = (\sigma^{-1}(2n-1),\sigma^{-1}(k))_{t_{a+2n}},$$ 
    $$(\sigma^{-1}(k),2n+1)_{t_{a+2n}} = (\sigma^{-1}(k),\sigma^{-1}(2n-1))_{t_{a+2n}}.$$
    In particular, from the previous equalities, we obtain
    \begin{eqnarray*}
        (\sigma^{-1}(2n-1),\sigma^{-1}(2n))_{t_{a+2n}} &=& (\sigma^{-1}(2n-1),\sigma^{-1}(2n))_{t_{a}}\\
&=&(\sigma^{-1}(2n-1),2n+2)_{t_{a}}\\
        &=& (\sigma^{-1}(2n-1),2n+1)_{t_{a}} \\
        &=&(\sigma^{-1}(2n-1),2n+1)_{t_{a+2n}}=\overrightarrow{e_{n-}} \\
        (\sigma^{-1}(2n-1),\sigma^{-1}(2n))_{t_{a+2n}} &=& (2n+1,\sigma^{-1}(2n))_{t_{a+2n}}\\
        &=&(2n+1,\sigma^{-1}(2n))_{t_{a}}  \\
        &=& (2n+1,2n+2)_{t_{a}} =\overrightarrow{e_{(n+1)+}} \\
        \overrightarrow{e_{(n+1)+}}  &=& (2n+1,2n+2)_{0} = (2n+1,2n+2)_{t_{a}} \\
        &=&(2n+1,2n+2)_{t_{a+2n}}\\
        &=& (2n+1,\sigma^{-1}(2n))_{t_{a+2n}} \\&=&  (2n+2,\sigma^{-1}(2n))_{t_{a+2n}}=\overleftarrow{e_{(n+1)-}} 
    \end{eqnarray*}
    That is, we obtain 
    \begin{eqnarray*}
        \overrightarrow{e_{n-}}  &=& \overrightarrow{e_{(n+1)+}}  = \overleftarrow{e_{(n+1)-}} ,\\
        \overrightarrow{e_{n-}}  &=& (\sigma^{-1}(2n-1),\sigma^{-1}(2n)).
    \end{eqnarray*}
    Analogously, 
    \begin{eqnarray*}
        \overleftarrow{e_{n-}}  &=& \overleftarrow{e_{(n+1)+}}  = \overrightarrow{e_{(n+1)-}} ,\\
        \overleftarrow{e_{n-}}  &=& (\sigma^{-1}(2n),\sigma^{-1}(2n-1)).
    \end{eqnarray*}
    That is, secants in $\mathcal{S'}-\mathcal{S}$ can be presented by secants in $\mathcal{S}$ and it completes the proof.
\end{proof}

\begin{remark}
    The proof of Lemma~\ref{lemma4} is based on the fact that new trisecants containing $2n+1$ or $2n+2$ are Type 1 or Type 2 trivial. The main contribution is to show that new generators containing $2n+1$ or $2n+2$ are presented by $(i,j)$ for $1\leq i\neq j \leq 2n$ in $SQ(K)$.
\end{remark}
\begin{theorem}\label{thm2}\rm{(Main theorem for knots)}
The secant-quandle $SQ(K)$ is a knot invariant, where $K:=\hat{\beta}$, $\beta\in B_{2n}$, i.e., if two links $K$ and $K'$ are equivalent, then $SQ(K)\cong SQ(K')$. 
\end{theorem}
\begin{proof}
Since the invariance of the braid part has already been established in Theorem \ref{thm1}, we need to discuss the invariance of $SQ(K)$ under the stable equivalence of Birman knots, i.e., $K_{2n}$ and $\sigma_{2n}$, according to Proposition \ref{prop2}.

\textbf{Case 1.} If $\beta'=\sigma_{2j}\sigma_{2j-1} \sigma_{2j+1}\sigma_{2j}\beta$, denote $K:=\hat{\beta}$ and $K':=\hat{\beta}'$ . All new trisecants $\mathcal{T}_{K'}-\mathcal{T}_{K}$ are listed in order as follows
\begin{small}
\begin{align*}
\bigcup_{ _{\phantom +} \sigma_{2j\phantom 1 }}     &  \{{\color{blue}(2j-1,2j,2j+1)},(p,2j,2j+1),(2j+1,2j,q),(2j+1,2j,2j+2)\}\\
\bigcup_{\sigma_{2j-1}}& \{(p,2j-1,2j+1),(2j+1,2j-1,q),(2j+1,2j-1,2j+2),{\color{blue}(2j+1,2j-1,2j)}\}\\
\bigcup_{\sigma_{2j+1}}&  \{{\color{blue}(2j-1,2j,2j+2)},(2j+1,2j,2j+2),(p,2j,2j+2),(2j+2,2j,q)\}\\
\bigcup_{ _{\phantom 1}\sigma_{2j\phantom + }}    &  \{(2j+1,2j-1,2j+2),(p,2j-1,2j+2),(2j+2,2j-1,q),{\color{blue}(2j+2,2j-1,2j)}\},
\end{align*}
\end{small}where $p< 2j-1<2j<2j+1<2j+2<q$, and $\cup_{\sigma_{2j}}$ means collecting all the new trisecants caused by $\sigma_{2j}$, as shown in Figure \ref{fig:8}. Note that these trisecants are either of Type 1 or  Type 2 trivial trisecants (see Definition \ref{def5}).
\begin{itemize}
    \item Type 1: {\small{\color{blue}$(2j-1,2j,2j+1)$, $(2j+1,2j-1,2j)$, $(2j-1,2j,2j+2)$, $(2j+2,2j-1,2j)$}}.
    \item Type 2: {\small $\{(p,2j,2j+1),(p,2j-1,2j+1)\}$, $\{(2j+2,2j,q),(2j+1,2j-1,q)\}$, $\{(p,2j,2j+2),(p,2j-1,2j+2)\}$, $\{(2j+2,2j,q),(2j+2,2j-1,q)\}$}.
    \item Type 3: {\small $\color{red}\{(2j+1,2j,2j+2)_{t_1},(2j+1,2j,2j+2)_{t_3}\}$, $\color{red}\{(2j+1,2j-1,2j+2)_{t_2},(2j+1,2j-1,2j+2)_{t_4}$}, where $t_1<t_2<t_3<t_4$.
\end{itemize}
As shown in Figure \ref{fig:8}, the first four pairs are Type 2 since \emph{the middle secant} $(2j-1,2j)=\overrightarrow{e_{j\pm}} $ (see Definition \ref{def5}). The Type 3 is analogous to Type 2 which can be canceled by calculation. Specifically, the original order of these trisecants are following
\begin{small}
\begin{align*}
\cdots &
{\color{red}(2j+1,2j,2j+2)_{t_1}} 
{(2j+1,2j-1,q)} 
{\color{red}(2j+1,2j-1,2j+2)_{t_2}} \\
&{\color{blue}(2j+1,2j-1,2j)}
{\color{blue}(2j-1,2j,2j+2)}
{\color{red}(2j+1,2j,2j+2)_{t_3}} \\
&{(p,2j,2j+2)} {(2j+2,2j,q)}
{\color{red}(2j+1,2j-1,2j+2)_{t_4}}\cdots
\end{align*}
\end{small}
Using $\Lambda_2$, we have
\begin{small}
\begin{align*}
\cdots &
{\color{red}(2j+2,2j,2j+1)} 
{\color{red}(2j+2,2j-1,2j+1)} 
{\color{blue}(2j,2j-1,2j+1)}\\
& {\color{blue}(2j-1,2j,2j+2)}
{\color{red}(2j+1,2j,2j+2)} 
{\color{red}(2j+1,2j-1,2j+2)}\cdots
\end{align*}
\end{small}
Since ${\color{blue}(2j,2j-1,2j+1)}$ and ${\color{blue}(2j-1,2j,2j+2)}$ are Type 1, thus removing them induces an isomorphism (Lemma \ref{lemma3}) and four trisecants remain;
{\footnotesize
\begin{align*}
{\color{red}(2j+2,2j,2j+1)_{t_1}} 
{\color{red}(2j+2,2j-1,2j+1)_{t_2}}
{\color{red}(2j+1,2j,2j+2)_{t_3}}
{\color{red}(2j+1,2j-1,2j+2)_{t_4}.}
\end{align*}}
It is straightforward to verify that any such secant $(j,k)_{t_1}$ satisfies $(j,k)_{t_1} = (j,k)_{t_4}$ (left to the readers) passing through Type 3, while $(2j+1,2j+2)_{t_1}$ requires special attention: since the strands $2j-1$ and $2j$ are connected to a critical point, once their orientation is fixed, the alternating signs of Type 3 induce these equations:
$$(2j+1,2j+2)_{t_4} = (2j+1,2j+2)_{t_1} * a *^{-1} a  * a*^{-1} a=(2j+1,2j+2)_{t_1},$$
where $a=(2j,2j+1)_0$ and $*=\circ,/$, as shown in Figure \ref{fig:thm2}. The equation $(2j+2,2j+1)_{t_1}=(2j+2,2j+1)_{t_4}$ holds by the same reasoning. 
\begin{figure}
    \centering
    \includegraphics[width=0.75\linewidth]{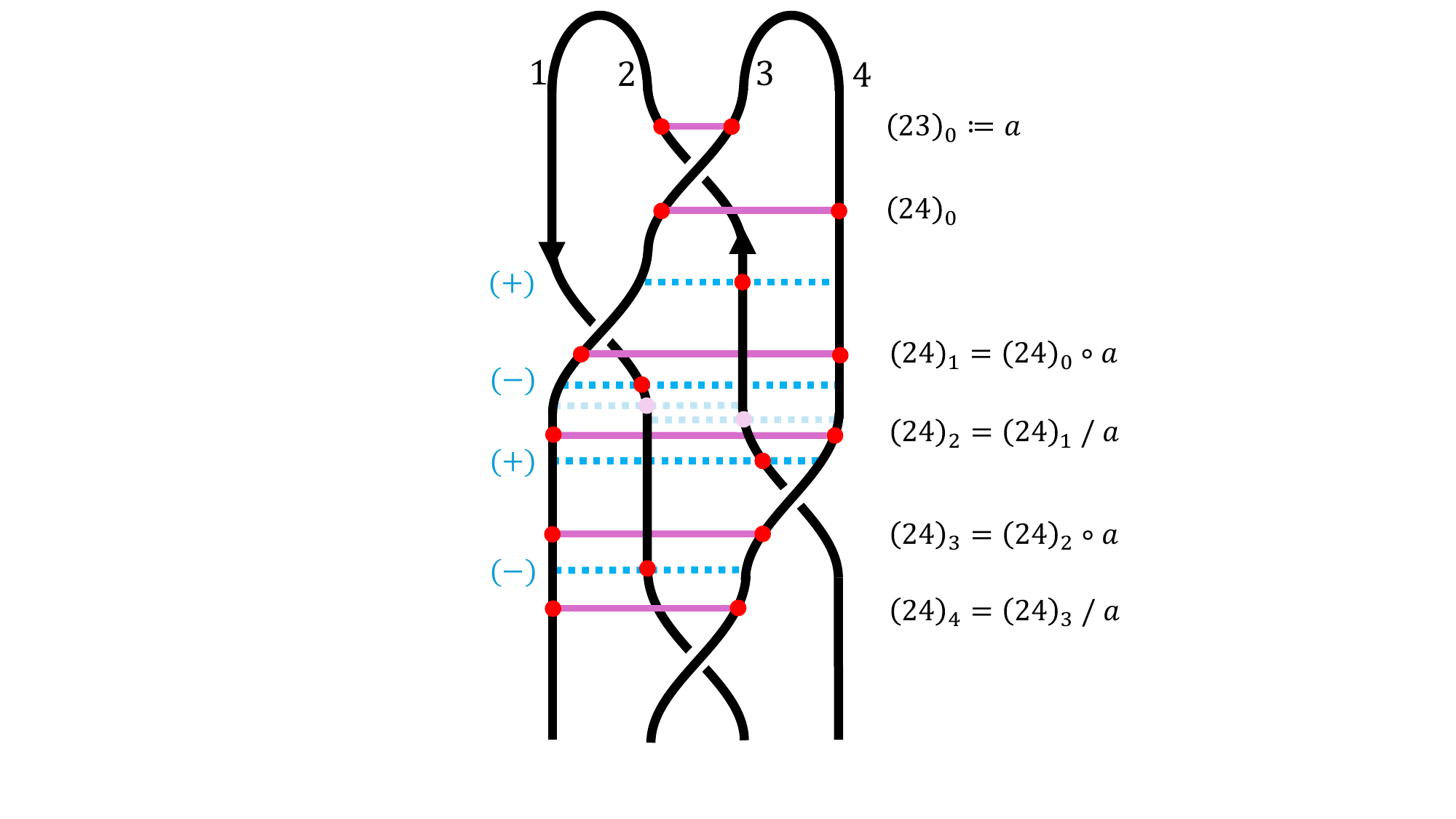}
    \caption{$\sigma_{2j}\sigma_{2j-1} \sigma_{2j+1}\sigma_{2j}$, taking $\sigma_{2}\sigma_{1} \sigma_{3}\sigma_{2}$ as a specific instance.}
    \label{fig:thm2}
\end{figure}

\textbf{Case 2.} If $\beta'=\sigma_{2}\sigma_{1}^2 \sigma_{2}\beta$. All new trisecants $\mathcal{T}_{K'}-\mathcal{T}_{K}$ are listed in order as follows
\begin{small}
\begin{align*}
\bigcup_{ \sigma_2} \{{\color{blue}(123)},{(32p)},{(324)}\}
\bigcup_{\sigma_1} \{{(31p)},{(314)},{\color{blue}(312)}\}\\
\bigcup_{\sigma_1} \{{(13p)},{\color{blue}(134)},{(132)}\}
\bigcup_{ \sigma_2} \{{(132)},{(23p)},{\color{blue}(234)}\}
\end{align*}
\end{small}
where $p>4$. We follow the same classification:
\begin{itemize}
    \item Type 1: ${\color{blue}(123)}, {\color{blue}(312)}, {\color{blue}(134)}, {\color{blue}(234)}$.
    \item Type 2: $\{(32p),(31p)\},\{(324),(314)\},\{(132),(132)\},\{(13p),(23p)\}.$
\end{itemize}
Using $\Lambda_2$, exchange $(423)$ with $(p13)$, then we obtain all the Type 2 trivial trisecants above. Note that the pair $\{(p31),(p32)\}$ is also Type 2, as shown in Figure \ref{fig:9}, since the surfaces $R_{1p}$ and $R_{2p}$ are connected while $(p31)$ and $(p32)$ are two adjacent trisecants in one film-frame $D=R_{1p}\cup R_{2p}$. 

\textbf{Case 3.}  If $\beta'=\beta\sigma_{1}$, all new trisecants $\mathcal{T}_{K'}-\mathcal{T}_{K}=\cup_{\sigma_1}\{(p12)\}$ where $p>2$. Obviously, each $(p12)$ is Type 1 trivial. 

\textbf{Case 4.} The case of  $\beta'=\beta\sigma_{2n}$ is proved in Lemma \ref{lemma4}.

That is, whatever the new trisecants comes from cases 1, 2, 3 and 4, $SQ(K)\cong SQ(K')$ according to Lemmas \ref{lemma3} and \ref{lemma4}, and it completes the proof.
\end{proof}

\begin{figure}[H]  % [H] 强制图片位置不浮动
    \centering
    \includegraphics[width=0.75\linewidth]{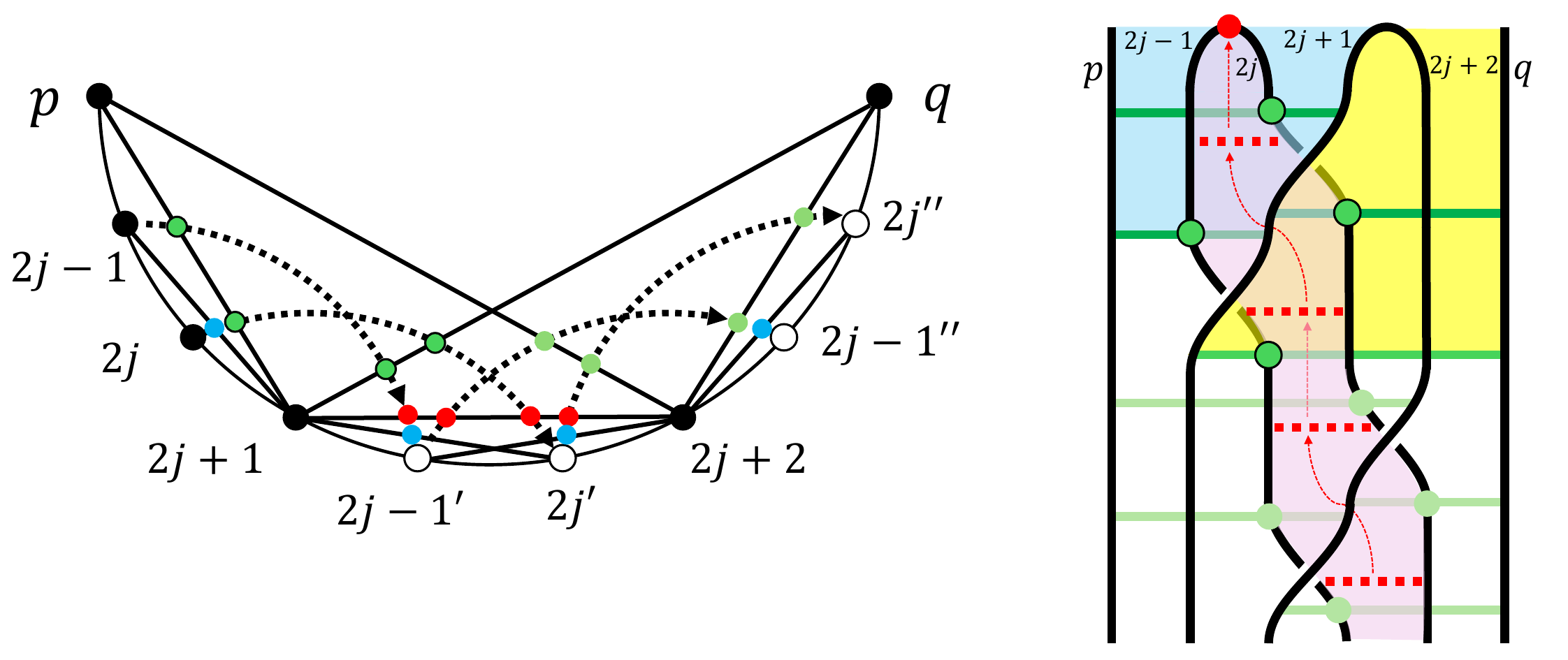}
    \caption{$\sigma_{2j}\sigma_{2j-1} \sigma_{2j+1}\sigma_{2j}$.}
    \label{fig:8}
\end{figure}

\begin{figure}[H]  % [H] 强制图片位置不浮动
    \centering
    \includegraphics[width=0.75\linewidth]{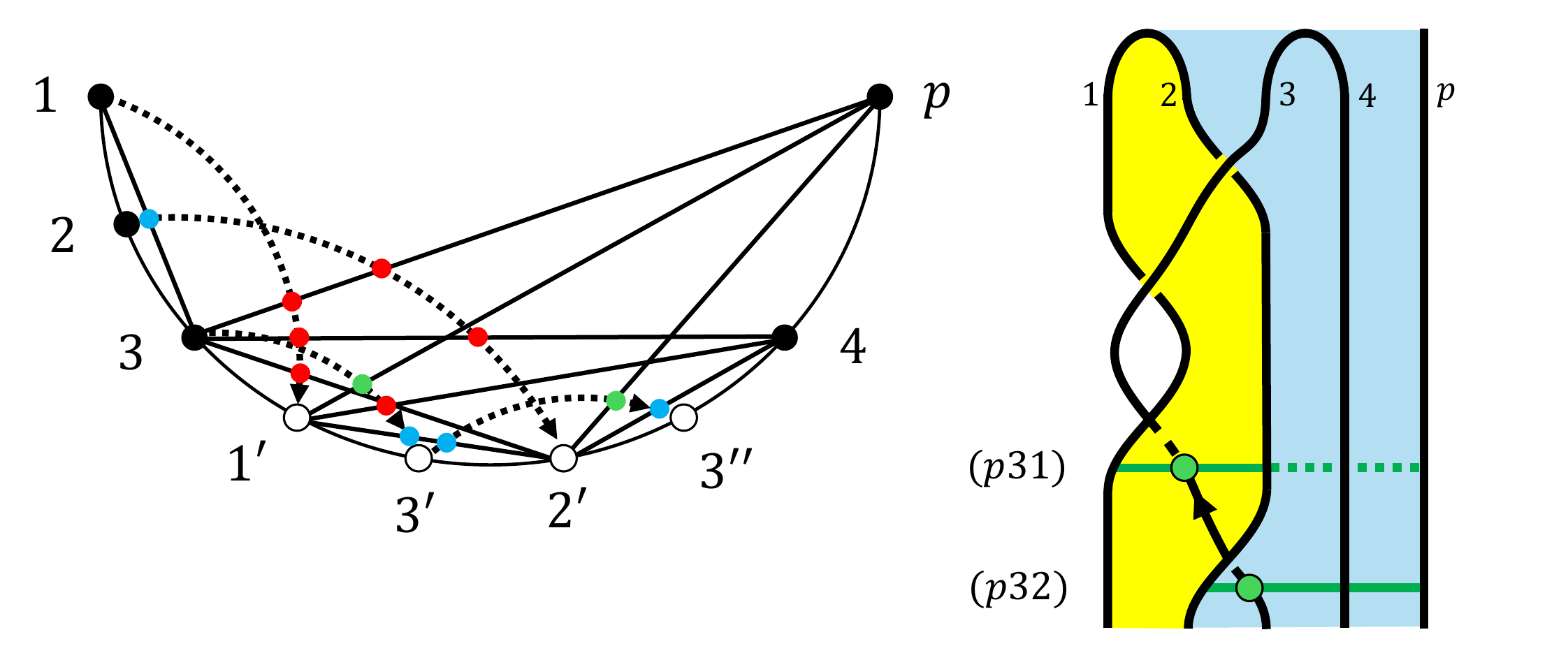}
    \caption{$\sigma_2 \sigma_1^2 \sigma_2$.}
    \label{fig:9}
\end{figure}
%%%%%%%%%%%%%%%%%%%%%%%%%%%%%%%%%%%%%%%%%%%%%%%%%%% Section 3
\begin{example} The secant quandles of the trivial knot and the trefoil are as follows.
\begin{itemize}
    \item $SQ(\text{trivial knot})=\Gamma\langle \overrightarrow{e},\overleftarrow{e}\rangle$.
    \item $SQ(\text{trefoil})=\Gamma\langle a\rangle$.
    %\item $SQ(\text{Hopf }\text{link})=\Gamma\langle \overleftrightarrow{e_{j\pm}},\overleftrightarrow{a}\mid \overrightarrow{e_{j-}}/\overrightarrow{a}=\overrightarrow{e_{j+}},\overleftarrow{e_{j-}}\circ \overleftarrow{a}=\overleftarrow{e_{j+}},\overleftrightarrow{a}\circ \overleftrightarrow{e_{j\pm}} =\overleftrightarrow{a},j=1,2 \rangle$, where both two components of Hopf link with positive orientation, and $\overrightarrow{e_{j+}}\in \mathcal{E}_K$ denotes $j$-th right vector secant at top level of $K$.
\end{itemize}
The calculation for the $SQ$ of the trivial knot is obvious.\\
The calculation for the $SQ$ of the trefoil is as follows. At the top of trefoil, we have $(12)_{\epsilon }=\overrightarrow{e_{1+}}$, $(34)_\epsilon=\overrightarrow{e_{2+}}$, $(13)_\epsilon =(14)_\epsilon=(24)_\epsilon=(23)_\epsilon=\overrightarrow{a}$ and their reverse secants. Then $\overrightarrow{a} = (23)_\epsilon$ is secant-homotopic to $(23)_{1-\epsilon}$ following the ruled surface $R_{23}$ since $R_{23}$ has no intersection with other strands. Since the pairs of strands $j_{1},j_{3}$ and $j_{2},j_{4}$ make minima, we obtain $(23)_{1-\epsilon} = (41)_{1-\epsilon}$. Then following the ruled surface $R_{14}$ we can show that $(41)_{1-\epsilon}$ is secant-homotopic to $(41)_{\epsilon} = \overleftarrow{a}$. Let us denote $a= \overrightarrow{a} =\overleftarrow{a}$.

We will show that $\overleftrightarrow{e_{1+}}=\overleftrightarrow{e_{2+}}=a$. Then the equalities $\overleftrightarrow{e_{1-}}=\overleftrightarrow{e_{2-}}=a$ follows from the symmetry of the trefoil. Using the 1st and the 2nd nontrivial trisecants $(132)_{t_1}$ and $(324)_{t_2}$, we obtain relations
$$\begin{array}{ll}
 (12)_{t_1}=(12)_{t_0}\circ (13)_{t_0},\quad
 (34)_{t_2}=(34)_{t_1} / (32)_{t_1}.
\end{array}$$
We can see that 
\begin{eqnarray*}
    (12)_{t_1}\sim (12)_{1-\epsilon} \sim a,\quad 
    (12)_{t_0}= \overrightarrow{e_{1+}},\quad
    (13)_{t_0} =\overrightarrow{a},
\end{eqnarray*}
and
\begin{eqnarray*}
    (34)_{t_2}\sim (34)_{1-\epsilon} \sim a,\quad
    (34)_{t_1}= \overrightarrow{e_{2+}},\quad
    (23)_{t_1} \sim (23)_{t_0} =\overrightarrow{a}.
\end{eqnarray*}
From the above observation, it follows that
$$\begin{array}{ll}
 \overrightarrow{e_{1+}}=a / a=a= a\circ a=\overrightarrow{e_{2+}}.
\end{array}$$
Analogously, we can show 
$$\begin{array}{ll}
 \overleftarrow{e_{1+}}=a=\overleftarrow{e_{2+}},
\end{array}$$
and therefore we obtain $SQ(\text{trefoil})=\Gamma\langle a\rangle$.
\begin{figure}  
    \centering
    \includegraphics[width=0.75\linewidth]{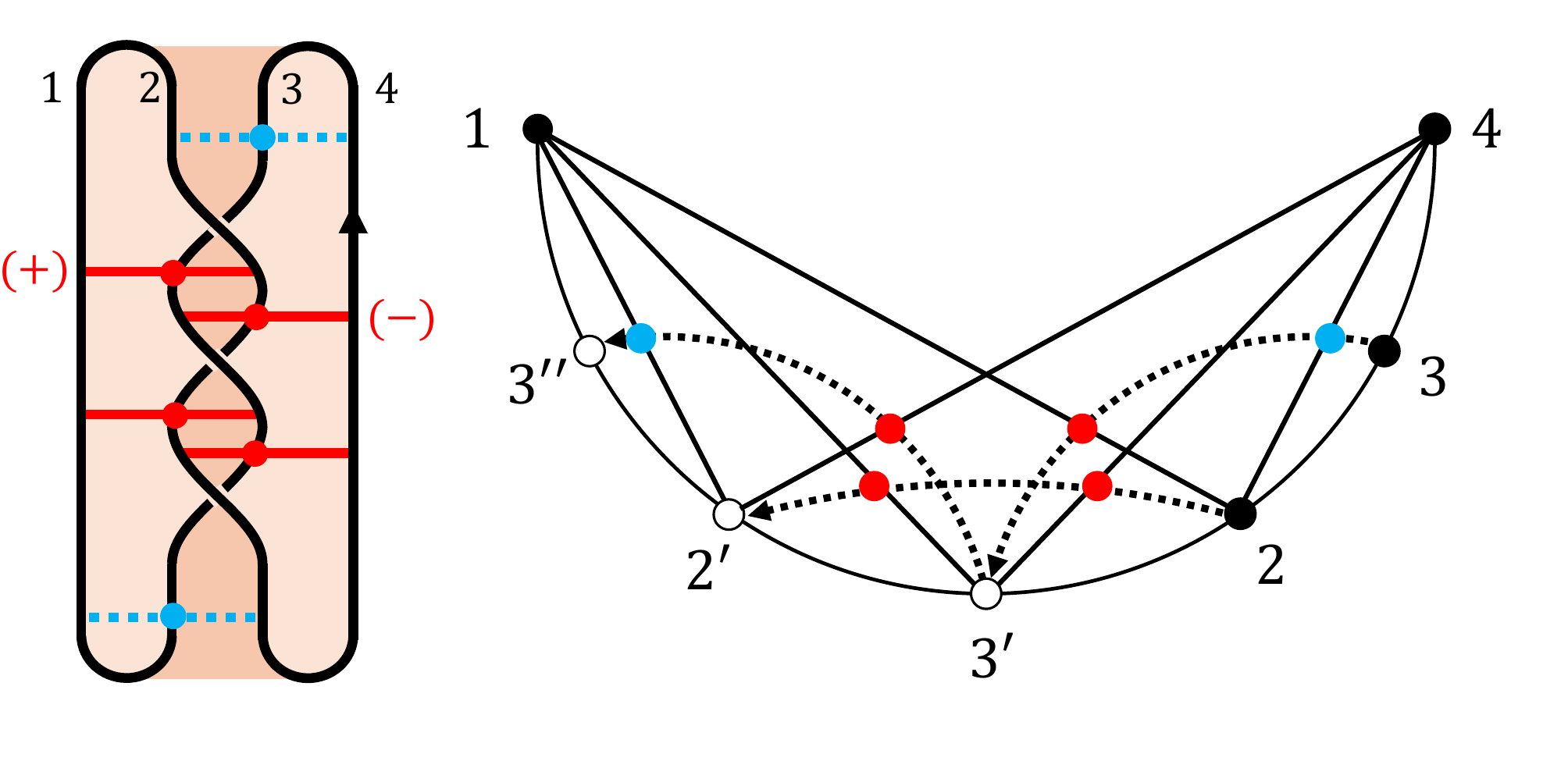}
    \caption{The $SQ$ of trefoil.}
    \label{fig:13}
\end{figure}
\end{example}
\section{Some properties of secant-quandle}

\begin{lemma}
    $SQ(K)\cong SQ(K^*)$, where $K^*$ is mirror image of $K$.
\end{lemma}
\begin{proof}
    For each trisecant $(j,k,l)\in \mathcal{T}_K$, then $\text{sign}(j,k,l)=\text{sign} (j,k,l)^*=\text{sign}(l',k',j')$, it induces the following isomorphism: 
$$(SQ(K),\circ) \rightarrow (SQ(K),/).$$
\end{proof}
\begin{lemma}
    $SQ(K)\cong SQ(-K)$, where $-K$ is inverted knot of $K$.
\end{lemma}
\begin{proof}
    For each trisecant $(j,k,l)\in \mathcal{T}_K$, we have $\text{sign}(-(j,k,l))=-\text{sign}(j,k,l)$. It follows that the inverse map $-:(SQ(K),\circ)\rightarrow (SQ(-K),/)$ is an isomorphism, being precisely the composition of the following maps:
$$\begin{CD}
   -: (SQ(K),\circ)@>\cong>> (SQ(K),/)@>\cong>> (SQ(-K),\circ)@>\cong>> (SQ(-K),/)
\end{CD}$$
\end{proof}
%\begin{lemma}
    %For $SQ(K\sqcup O)$ $(i,j)$ and $(i,2n+1)$ are not independent, because $(i,j)$ acts on $(i,2n+1)$ But $(2n+1,i)$ is independent with $(j,i)$. It means, $SQ(K\sqcup O) = Q\sqcup \{(2n+1,i)| (2n+1,i,j)\}$.
%\end{lemma}
\begin{theorem}\label{thm3}
    There is a monomorphism from $Q(K)$ to $SQ(K\sqcup O)$.
\end{theorem}
\begin{proof}
    Let $K=\hat\beta$, $K':=K\sqcup O=\hat \beta'$, where $\beta\in B_{2n}$, and $\beta'\in B_{2n+2}$ is the image of natural inclusion $B_{2n} \hookrightarrow B_{2n+2}$ . We obtain that $\mathcal{S}_{K'} = \mathcal{S}_{K}\sqcup \{\text{secants}~(i,2n+1),(i,2n+2), (2n+1,i),(2n+2,i), (2n+1,2n+2), (2n+2,2n+1)\}$ and $ \mathcal{T}_{K'} = \mathcal{T}_{K}\sqcup \{\text{trisecants in [0,1] containing points } 2n+1,2n+2 \}$. In particular $ \mathcal{T}_{K'} - \mathcal{T}_{K}$ consists of trisecants in the form of $(i,j,2n+1)$, $(i,j,2n+2)$, $(2n+1,i,j)$ and $(2n+2,i,j)$.
    According to Part 1 of the proof in Lemma \ref{lemma4} (it is identical with this case), we have: 
    \begin{itemize}
        \item $(k,l)$ remains under actions from trisecants containing $2n+1$ or $2n+2$ for $1\leq k,l\leq2n$.
        \item $(i,2n+1)_{t_{s}} = (i,2n+2)_{t_{s}}$, for $s=0,\dots, a$.
        \item $(2n+1,i)_{t_{s}} = (2n+2,i)_{t_{s}}$, for $s=0,\dots, a$.
    \end{itemize}
    That is, it is possible to reduce generators and relations to
    \begin{eqnarray*}
        \mathcal{S'}_{K'} &=& \mathcal{S}_{K}\sqcup \{\text{secants}~(i,2n+1),(2n+1,i), (2n+1,2n+2), (2n+2,2n+1)\}\\
        \mathcal{T'}_{K'} &=& \mathcal{T}_{K}\sqcup \{(i,j,2n+1),(2n+1,i,j) \}.
    \end{eqnarray*}
    Let $\mathcal{S}_{0} = \{\text{secants}~(i,2n+1),(2n+1,i)\}$ and $\mathcal{S}_{1}= \mathcal{S'}_{K'} - \mathcal{S}_{0}$.
    Note that $(i,j,2n+1)$ and $(2n+1,i,j)$ provide relations:
    $$\begin{array}{ll}
        (i,2n+1)_{t_{s+1}} = (i,2n+1)_{t_{s}}\circ (i,j)_{t_{s}},\\
        (2n+1,i)_{t_{s+1}}= (2n+1,i)_{t_{s}}/ (2n+1,j)_{t_{s}},\\
        (j,2n+1)_{t_{s+1}} = (j,2n+1)_{t_{s}}/(j,i)_{t_{s}},\\
        (2n+1,j)_{t_{s+1}}= (2n+1,j)_{t_{s}}\circ  (2n+1,i)_{t_{s}}.
    \end{array}$$
Let $\mathcal{T}_{0}=\{(2n+1,i)_{t_{s+1}}= (2n+1,i)_{t_{s}}/ (2n+1,j)_{t_{s}},(2n+1,j)_{t_{s+1}}= (2n+1,j)_{t_{s}}\circ  (2n+1,i)_{t_{s}}\}$.
Since $\mathcal{T}_{K}$ has not relations containing $(2n+1,i)$, we obtain
$$\mathcal{T'}_{K'} = \mathcal{T}_{0} \sqcup \mathcal{T}_{1}, $$
where $\mathcal{T}_{1} = \mathcal{T'}_{K'}-\mathcal{T}_{0}$.
It follows that 
$$SQ(K\sqcup O) = \Gamma\langle\mathcal{T}_{0} \mid \mathcal{S}_{0} \rangle*\Gamma\langle\mathcal{T}_{1} \mid \mathcal{S}_{1} \rangle,$$
where `$*$' is the free product of two quandles.

On the other hand, by a stereographic projection at each height $t \in [-0,5,1.5]$ with the fixed point $2n+1$ onto a plane, we can obtain a knot diagram $D$ which is the plat closure of a braid in $B_{2n}$, see Figure~\ref{fig:stereographic-proj}.
\begin{figure}
    \centering
\includegraphics[width=0.5\linewidth]{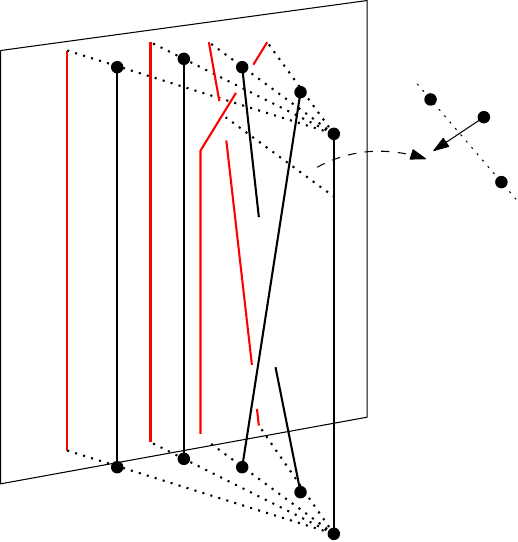}
    \put(-60,170){{\small $\Pi$}}
    \put(-170,167){{\footnotesize $(2n+1,i)$}}
    \put(-145,175){{\footnotesize $(2n+1,j)$}}
    \put(-110,169){{\footnotesize $(2n+1,l)$}}
    \put(-80,158){{\footnotesize $(2n+1,k)$}}
      \put(-60,140){{\small n+1}}
      \put(-30,120){{\small $P_{n+1}$}}
      \put(-30,150){{\small $P_{k}$}}
      \put(-10,146){{\small $P_{l}$}}
    \caption{We project arcs of a braid onto a plane $\Pi$ so that a point $P_{j}$ which is the intersection of the arc $j$ with a horisontal plane, is proejcted onto a plane from the point $P_{n+1}$. Then we obtain a braid diagram. In particular, each arc corresponds to a secant $(2n+1,j)$ and, when two points $P_{k},P_{l}$ are placed on a line with $P_{n+1}$, we obtain a classical crossing.}
    \label{fig:stereographic-proj}
\end{figure}
By abuse of notations, we denote the diagram by $\hat{\beta}$. From the diagram $\hat{\beta}$ we can obtain $Q(K)$ with the quandle presentation generated by arcs with relations coming from classical crossings. Each point on an arc $\alpha$ of $\hat{\beta}$ can be associated with $(2n+1,i)$ with respect to the construction of the knot diagram $\hat{\beta}$. Moreover, each crossing of $D$ comes from a trisecant $(2n+1,i,j)$ or $(i,j,2n+1)$ with respect to the sign of crossings. It induces a quandle homomorphism $\phi$ from $Q(K)$ to $\Gamma\langle \mathcal{T}_{0} \mid \mathcal{S}_{0} \rangle \subset SQ(K\cup O)$. 
It is clear that there exists one-to-one correspondence between generators of $Q(K)$ (relations of $Q(K)$, resp.) and $\mathcal{S}_{0}$ ($\mathcal{T}_{0}$, resp). Therefore, $Q(K)$ is isomorphic to $\Gamma\langle \mathcal{T}_{0} \mid \mathcal{S}_{0} \rangle$, which is subquandle of $SQ(K\sqcup O)$.
\end{proof}
%\begin{figure}
%    \centering
%    \includegraphics[width=1\linewidth]{figure108.pdf}
%    \caption{$SQ(K\sqcup O)$}
%    \label{fig:108}
%\end{figure}
From Theorem 3, we see that $SQ(K\sqcup O)$ is a free product of $Q(K)$ and a subquandle $R$ of $SQ(K \sqcup O)$. In the proof, classical crossings of a diagram are associated with trisecants. At the same time, there are trisecants that are not associated with any classical crossing of the diagram, but present ``hidden'' classical crossings. By the construction, the subquandle $R$ contains information from these hidden crossings. We expect that the additional part $R$ will provide topological properties of links.

\section*{Acknowledgments}
We are deeply grateful to I. N. Nikonov,  S. Özlem for their valuable suggestions and patient discussions.  Thank L. H. Kauffman for his invitation to show this work at his Quantum Topology Conference.

\end{document}